\documentclass[a4paper]{article}

\usepackage{latexsym}
\usepackage{amssymb}
\usepackage{amsmath}
\usepackage{amsthm}  % œrodowisko proof (proof environment)
\usepackage{amssymb,longtable}
\usepackage{graphics,graphicx}
\usepackage{eufrak}
\usepackage{bbm}
\usepackage{color}

\usepackage{caption2}   % zmiana podpisu pod rysunkiem - polecenie caption

\newtheorem{Property}             {{\textbf Property}}[section]
\newtheorem{Lemma}[Property]      {\textbf Lemma}
\newtheorem{Theorem}[Property]    {\textbf Theorem}
\newtheorem{Proposition}[Property]{\textbf Proposition}
\newtheorem{Remark}[Property]     {\textbf Remark}
\newtheorem{Corollary}[Property]  {\textbf Corollary}
\newtheorem{Example}[Property]    {\textbf Example}

\long\def\symbolfootnote[#1]#2{\begingroup%
\def\thefootnote{\fnsymbol{footnote}}\footnote[#1]{#2}\endgroup}
\title{Invariants of plane curve singularities and Newton diagrams}
\author{by Pierrette Cassou-Nogu\`{e}s and  Arkadiusz P\l{}oski}
\date{}

\begin{document}

\maketitle

\begin{center}
To Professor Kamil Rusek on his $65^{\textrm{th}}$ birthday
\symbolfootnote[0]{\emph{$2000$ Mathematics Subject Classification:} Primary 32S55; Secondary 14H20.}
\symbolfootnote[0]{\emph{Key words and phrases:} plane local curve, intersection multiplicity, Milnor number,}
\symbolfootnote[0]{Newton's transformation, Newton diagram, equisingularity.}
\end{center}
\bigskip

\begin{abstract}
We present an intersection-theoretical approach to the invariants of plane curve singularities
$\mu$, $\delta$, $r$ related by the Milnor formula $2\delta=\mu+r-1$.
Using Newton transformations we give formulae for  $\mu$, $\delta$, $r$ which imply planar versions
of well-known theorems on nondegenerate singularities.
\end{abstract}

\section*{0 \ \ Introduction}
\label{sec:0.Introduction}

The goal of this paper is to present an elementary, 
intersection-theore\-ti\-cal approach to the local invariants of plane curve singularities.
We study in detail three invariants:
the Milnor number $\mu$, the number of double points $\delta$ and the number $r$ of branches of a local plane curve.
The technique of Newton diagrams plays an important part in the paper.
It is well-known that Newton transformations which arise in a natural way
when applying the Newton algorithm provide a useful tool for calculating invariants of singularities.

The formulae for the Milnor number in terms of  Newton diagrams and Newton transformations
presented in the paper grew out of our discussion on  Eisenbud-Neumann diagrams. 
They have counterparts in toric geometry of plane curve singularities and imply 
in the case of two dimensions theorems due to Kouchnirenko, Bernstein and Khovanski.\\

The contents of the article are:
\begin{enumerate}
\item {Plane local curves}
\item {The Milnor number: intersection theoretical approach}
\item {Newton diagrams and power series}
\item {Newton transformations and factorization of power series}
\item {Newton transformations, intersection multiplicity and the Milnor number}
\item {Nondegenerate singularities and equisingularity}\\
\end{enumerate}

\section{Plane local curves}
\label{s1}

Let $\mathbbm{C}\{X,Y\}$ be the ring of convergent complex power series in variables $X$, $Y$.
For any nonzero power series 
$f = \sum c_{\alpha\beta} X^{\alpha} Y^{\beta}$ we put
 $\mbox{supp}f = \{(\alpha, \beta) \in \mathbbm{N}^2: c_{\alpha\beta} \ne 0\}$,
 $\mbox{ord}f = \inf\{\alpha + \beta: (\alpha, \beta) \in \mbox{supp}f \}$ and 
$\mbox{in}f = \sum c_{\alpha\beta} X^\alpha Y^\beta$ with summation over 
$(\alpha, \beta) \in \mathbbm{N}^2$ such that $\alpha + \beta =\mbox{ord}f $.

We put by convention $\mbox{ord}0  = +\infty$, $\mbox{in}0=0$. 
We call $c_{00}$ the constant term of the 
power series $f$. 
The power series without constant term form the unique maximal ideal of 
$\mathbbm{C} \{X,Y\}$. A power series is a unit if and only if its constant term is nonzero. We write $g=f\cdot\mbox{unit}$  if there is a unit $u$ such that $g=fu$ in $\mathbbm{C} \{X,Y\}$. We say also that $f$ and $g$ are associated.
Let $f \in \mathbbm{C}\{X,Y\}$ be a nonzero power series without constant term. A \emph{local (plane) curve} 
$f=0$ is defined to be the ideal generated by $f$ in $\mathbbm{C}\{X,Y\}$. We say that a local curve $f=0$ 
is \emph{irreducible (reduced)} if $f \in \mathbbm{C}\{X,Y\}$ is irreducible ($f$ has no multiple factors). The irreducible curves are also called \emph{branches}. If $f=f_{1}^{m_1} \ldots f_{r}^{m_r}$ with non-associated irreducible factors $f_i$ then we refer to $f_i=0$ as the branches or components of $f=0$. 
We say that a curve $f=0$ is \emph{singular (nonsingular)} if $\mbox{ord}f>1$ ($\mbox{ord} f=1$). We call $\mbox{ord}f$ the \emph{multiplicity of the curve $f=0$}. The lines defined by the equation $\mbox{in}f=0$ are the \emph{tangent lines} (in short: tangents) to the curve $f=0$. 

Let $\widetilde{X}, \widetilde{Y}$ be new variables. 
A local system of coordinates $\Phi$ is a pair of power series 
$\Phi(\widetilde{X},\widetilde{Y}) = (a\widetilde{X} + b\widetilde{X} + \cdots, c\widetilde{X}+d\widetilde{Y}+\cdots) $ where $ad-bc\ne 0$  and the dots denote terms of order higher than $1$ in $\widetilde{X}, \widetilde{Y}$. The map $f \rightarrow f \circ \Phi$ is an isomorphism of the rings $\mathbbm{C} \{X,Y\}$ and $\mathbbm{C} \{\widetilde{X},\widetilde{Y}\}$.

For any power series $f,g \in \mathbbm{C} \{X,Y\}$ we define the \emph{intersection multiplicity} or \emph{intersection number} $i_0(f,g)$ by putting
$$i_0 (f,g) = \dim_\mathbbm{C} \mathbbm{C} \{X,Y\}/(f,g)$$ 
where $(f,g)$ is the ideal of $\mathbbm{C} \{X,Y\}$ generated by $f$ and $g$. If $f,g$ are nonzero power series without constant term then $i_0(f,g) < +\infty$ if and only if the curves $f=0$ and $g=0$ have no common branch. The following properties are basic 
\begin{enumerate}
\item $i_0(f,g)$ depends only on the ideal $(f,g)$. In particular
$i_0 (f,g) = i_0 (g,f)$ and $i_0 (f, g+kf)=i_0(f,g)$.
\item If $\Phi$ is a local system of coodinates then $i_0 (f\circ \Phi, g\circ \Phi) = i_0 (f,g)$.
\item $i_0 (f,gh) = i_\circ (f,g) + i_\circ (f,h)$.
\end{enumerate}

Let $t$ be a variable. A \emph{parametrization} is a pair $(x(t),y(t)) \in \mathbbm{C} \{t\}^2$ of power series without constant term such that $x(t) \ne 0$ or $y(t) \ne 0$ in $\mathbbm{C} \{t\}$. Two parametrizations $(x(t),y(t))$ and $(\widetilde{x}(\widetilde{t}),\widetilde{y}(\widetilde{t}))$  are equivalent if there is a power series $\tau(t)\in\mathbb{C}(t)$, $\mbox{ord}\tau=1$ such that $x(t)=\widetilde{x}(\tau(t))$, $y(t)=\widetilde{y}(\tau(t))$.
A \emph{parametrization} $(x(t),y(t)) \in \mathbbm{C}\{t\}^2$ \emph{is good} if there is no parametrization 
$(x_1(t_1),y_1(t_1))\in \mathbbm{C} \{t_1\}^2$ such that $x(t) = x_1 (\tau_1(t))$, $y(t) = y_1(\tau_1(t))$ for
a power series $\tau_1(t)$ such that $\mbox{ord}\tau_1(t) > 1$. 

A parametrization $(x(t),y(t))$ is a \emph{Puiseux' parametrization} if it is good and $x(t) = t^n$ for
an integer $n>0$. One checks that a parametrization $(t^n,y(t))$ is a Puiseux' parametrization 
if and only if $\mbox{gcd}(n,\mbox{supp}y(t))=1$.

For any branch $f=0$ there is a unique up to equivalence good parametrization $(x(t),y(t))$ such that $f(x(t),y(t))=0$. If $n=i_0(f,x)<+\infty$ then it is equivalent to a Puiseux' parametrization $(t^n, y(t))$.
On the other hand for any parametrization $(x(t),y(t))$ there is a unique branch $f=0$ such that $f(x(t),y(t))=0$.

 We have the following important property
\begin{enumerate}
\item [4.] If $(x(t),y(t))$ is a good parametrization of the branch $f=0$ then $i_0(f,g) =\mbox{ord}g(x(t),y(t))$.
\end{enumerate}
which implies
\begin{enumerate}
\item [5.] Let $f=0$ be a branch. Then for any power series $g,h \in \mathbbm{C} \{X,Y\}$:\\
$i_0 (f,g+h) \geq \inf\{i_0(f,g), i_0(f,h)\}$ with equality if $i_0(f,g) \ne i_0(f,h)$. 
\end{enumerate}
Suppose that $f=0$ is a branch and consider 
$$
\Gamma(f) = \{i_0 (f,g): g \in \mathbbm{C} \{X,Y\} \mbox{ runs over all series such that $f$ does not divide $g$}\}.
$$
Clearly $0\in \Gamma(f)$ and $a,b \in \Gamma(f) \Rightarrow a+b \in \Gamma(f)$ since the intersection number is additive. We call $\Gamma(f)$ the semigroup of the branch $f=0$. Note that $\Gamma(f)=\mathbbm{N}$ if and only if the branch $f=0$ is nonsingular.

Two reduced curves $f=0$ and $g=0$ are \emph{equisingular} if and only if there are factorizations $f=f_1\cdots f_r$ and $g=g_1\cdots g_r$ with the same numbers $r>0$ of irreducible factors $f_i$ and $g_i$ such that
\begin{itemize}
\item $\Gamma(f_i) = \Gamma(g_i)$ for all $i=1,\ldots,r$,
\item $i_0(f_i,f_j) = i_0 (g_i,g_j)$ for $i,j = 1,\ldots,r$.
\end{itemize}
The bijection $f_i\mapsto g_i$ will be called equisingularity bijection.
In particular two branches are equisingular if and only if they have the same semigroup. A function defined on the set of reduced curves is an \emph{invariant} if it is constant on equisingular curves. The multiplicity and the number of branches of a plane local curve are invariants.\\

\noindent
\textbf{Notes}\\
The proofs omitted in this section are given in \cite{GrLoSh06}. 
A beautiful introduction to the subject is given in \cite{Te91}. 
The book \cite{BriKn86} is very accesible and contains historical information. 
For the systematic treatment of plane curve singularities see \cite{Ca00}, \cite{JoPf00}, Chapter $5$, 
and \cite{Wall04}.\\

\section{The Milnor number: intersection theoretical approach}
\label{s2}

For every power series $f\in \mathbbm{C} \{X,Y\}$ without constant term we define the Milnor number $\mu_0(f)$ by
putting
$$\mu_0(f) = i_0\bigg(\frac{\partial f}{\partial X}, \frac{\partial f}{\partial Y}\bigg).$$

\begin{Property}
\label{wl1}
We have $\mu_0 (f) = + \infty$ if and only if $f$ has a multiple factor in $\mathbbm{C} \{X,Y\}$.
\end{Property}
\smallskip

\begin{proof}
If $f=h^2g$ in $\mathbbm{C} \{X,Y\}$, $\mbox{ord}h > 0$ then 
$\frac{\partial f}{\partial X} = 2h\frac{\partial h}{\partial X}g + h^2\frac{\partial g}{\partial X}$ and\linebreak  
$\frac{\partial f}{\partial Y} = 2h\frac{\partial h}{\partial Y}g + h^2\frac{\partial g}{\partial Y}$. 
Thus the derivatives $\frac{\partial f}{\partial X}, \frac{\partial f}{\partial Y}$ have a common factor $h$ of positive order and 
$\mu_0 (f) = i_0 \Big(\frac{\partial f}{\partial X}, \frac{\partial f}{\partial Y} \Big) = +\infty$. 
Now suppose that $i_0 \Big(\frac{\partial f}{\partial X}, \frac{\partial f}{\partial Y} \Big) = +\infty$. 
Then there exists an irreducible divisor $h$ of the derivatives $\frac{\partial f}{\partial X}, \frac{\partial f}{\partial Y}$. 
We claim that $h$ divides $f$: if $(x(t),y(t))$ is a parametrization of the branch $h=0$ then 
$\frac{d}{dt} f(x(t),y(t)) = \frac{\partial f}{\partial X} (x(t),y(t)) x'(t) + \frac{\partial f}{\partial Y}(x(t),y(t)) y'(t) = 0$ 
in $\mathbbm{C}\{t\}$. 
Therefore $f(x(t),y(t))=0$ and $h$ divides $f$.
From irreducibility of $h$ it follows that $\mbox{ord} h(X,0)=\mbox{ord} h$ or $\mbox{ord} h (0,Y) = \mbox{ord}h$.
Suppose that $\mbox{ord} h(0,Y) = \mbox{ord} h$. 
Thus $\mbox{ord} \frac{\partial h}{\partial Y} = \mbox{ord} h - 1$ and the power series $h$ and $\frac{\partial h}{\partial Y}$ are coprime in $\mathbbm{C}\{X,Y\}$. 
Write $f=hg$. Whence
$\frac{\partial f}{\partial Y} = \frac{\partial h}{\partial Y} g + h \frac{\partial g}{\partial Y}$ and $h$ divides 
$\frac{\partial h}{\partial Y}g$. 
Therefore $h$ divides $g$ and $h$ is a multiple factor of $f$.\\
\end{proof}

\begin{Property}
For any local system of coordinates $\Phi: \mu_0 (f \circ \Phi) = \mu_0 (f)$. 
\end{Property}
\smallskip

\begin{proof}
Since $\mbox{Jac}\Phi(0,0)\ne 0$ the ideals 
$\Big(\frac{\partial}{\partial \widetilde{X}} (f \circ \Phi), \frac{\partial}{\partial \widetilde{Y}} (f \circ \Phi)\Big)$ and\\
$\Big(\frac{\partial f}{\partial X} \circ \Phi, \frac{\partial f}{\partial Y} \circ \Phi\Big)$
are equal. 
Thus we get $\mu_0 (f\circ \Phi) = i_0 \Big(\frac{\partial f}{\partial X} \circ \Phi, \frac{\partial f}{\partial Y} \circ \Phi \Big) =$\\
$=i_0 \Big(\frac{\partial f}{\partial X}, \frac{\partial f}{\partial Y} \Big) =\mu_0 (f)$.
\end{proof}
\smallskip

The following lemma, due to Teissier (\cite{Te73}, Chap. II, Th\'{e}or\`{e}me 5; \cite{Te73$^{bis}$}, Chap. II, Prop. 1.2) is basic for us. 
It is a particular case of a formula proved in \cite{Te73} in the case of hypersurfaces.\\

\begin{Lemma}[Teissier's lemma]
 Let $f\in\mathbb{C}\{X,Y\}$, $f(0,0)=0$ be such that $f(0,Y)\ne 0$. Then we have
$$i_0 \left(f,\frac{\partial f}{\partial Y}\right)=\mu_0(f)+i_0(f,X)-1.$$
\end{Lemma}
\smallskip

\begin{proof}
It is easy to check using Property \ref{wl1} that $i_0 \left(f,\frac{\partial f}{\partial Y}\right)=+\infty$ if and only if $\mu_0(f)=+\infty$. 
Suppose that $\mu_0(f)<+\infty$ and $\frac{\partial f}{\partial Y}(0,0)=0$ 
(if $\frac{\partial f}{\partial Y}(0,0)\ne 0$ then the lemma is obvious).
Write $\frac{\partial f}{\partial Y}=g_1\cdots g_m$ with irreducible $g_i\in\mathbb{C}\{X,Y\}$.
Let $(x_i(t_i),y_i(t_i))\in\mathbb{C}\{t_i\}^2$ be a good parametrization of the branch $g_i=0$.
Differentiating and taking orders give $\mbox{ord} f(x_i(t_i),y_i(t_i))=\mbox{ord}\frac{\partial f}{\partial X}(x_i(t_i),y_i(t_i))+\mbox{ord}x_i(t_i)$ that is $i_0(f,g_i)=i_0\left(\frac{\partial f}{\partial X},g_i\right)+i_0(X,g_i)$ for $i=1,\ldots,m$. 
Summing up the obtained equalities we get 
$i_0\left(f,\frac{\partial f}{\partial Y}\right)=\mu_0(f)+i_0\left(X,\frac{\partial f}{\partial Y}\right)$
and the lemma follows since
$i_0\left(X,\frac{\partial f}{\partial Y}\right)=i_0(X,f)-1$. 
\end{proof}
\smallskip

\begin{Property}
\label{2_4}
Let $f\in\mathbb{C}\{X,Y\}$, $f(0,0)=0$ be a power series without multiple factors. Then
\begin{enumerate}
\item[$(i)$]  if $g=f\cdot$unit then $\mu_0(g)=\mu_0(f)$,
\item[$(ii)$] if $f=f_1\cdots f_m$, $f_i(0)=0$ and $f_i$ are pairwise coprime then
$$\mu_0(f)+m-1=\sum_{i=1}^{m}\mu_0(f_i)+2\sum_{1\le i<j\le m}i_0(f_i,f_j)$$
\end{enumerate}
\end{Property}
\smallskip

\begin{proof}
We may assume that $f(0,Y)\ne 0$ in $\mathbb{C}\{Y\}$.\\
$(i)$ It is easy to check that $i_0\left(g,\frac{\partial g}{\partial Y}\right)=i_0\left(f,\frac{\partial f}{\partial Y}\right)$ and $i_0(g,X)=i_0(f,X)$. Then $\mu_0(g)=\mu_0(f)$ by Teissier's lemma.\\
$(ii)$ The basic properties of intersection multiplicity give
\begin{flushleft}
$\displaystyle i_0\left(f,\frac{\partial f}{\partial Y}\right)=\sum_{i=1}^{m}i_0\left(f_i,\frac{\partial f_i}{\partial Y}\right)+2\sum_{1\le i<j\le m}i_0(f_i,f_j),$\\
$\displaystyle i_0(f,X)=\sum_{i=1}^{m}i_0(f_i,X)$
\end{flushleft}
Then we use Teissier's lemma.
\end{proof}
\smallskip

\noindent
In what follows we need a lemma due to Jung (\cite{Ju23}, Zehntes Kapitel, \S 4, S. 181).\\

\begin{Lemma}[Jung's lemma]
\label{Jung's lemma}
Let $f(X,Y)=Y^n+a_1(X)Y^{n-1}+\cdots+a_n(X)\in\mathbb{C}\{X\}[Y]$ be a distinguished irreducible polynomial of degree $n>1$.
Let \emph{$D(X)=\mbox{disc}_Yf(X,Y)$} be the discriminant of $f$ Then \emph{$\mbox{ord}D(X)\equiv n-1\pmod{2}$}.
\end{Lemma} 
\smallskip

\begin{proof}
Let $\varepsilon_0$ be a primitive $n$-th root of unity. Then by Puiseux' Theorem $\displaystyle f(t^n,Y)=\prod_{k=0}^{n-1}{(Y-y(\varepsilon_0^kt))}$, where $y(t)\in\mathbb{C}\{t\}.$\\
Let $\displaystyle V_n(T_1,\ldots,T_n)=\prod_{1\leq i<j\leq n}{(T_i-T_j)}.$\\
Then \ $D(t^n)=\textrm{disc}_Yf(t^n,Y)=v_n(t)^2$ \ where \ $v_n(t)=V_n(y(t),y(\varepsilon_0 t),\ldots,y(\varepsilon_0^{n-1}t))$.
It is easy to check that $v_n(\varepsilon_0 t)=(-1)^{n-1}v_n(t)$. Let us distinguish two cases.\\[0.1cm]
Case 1. $n-1\equiv 0 \pmod{2}$.\\
From $v_n(\varepsilon_0 t)=v_n(t)$ we get $v_n(t)\in\mathbb{C}\{t^n\}$ i. e. $v_n(t)=d(t^n)$ where $d(X)\in\mathbb{C}\{X\}$. Thus $D(X)=d(X)^2$ and we get $\mbox{ord}D(X)\equiv 0\pmod{2}$.\\[0.1cm]
Case 2. $n-1\equiv 1 \pmod{2}$.\\
Then we have $v_n(\varepsilon_0 t)=-v_n(t)$ which implies $v_n(t)\in t^{\frac{n}{2}}\mathbb{C}\{t^n\}$
i. e. $v_n(t)=t^{\frac{n}{2}}d_1(t^n)$ where $d_1(X)\in\mathbb{C}\{X\}$. Thus $D(X)=Xd_1(X)^2$ and  $\mbox{ord}D(X)\equiv 1\pmod{2}$.\\
Summing up we get  $\mbox{ord}D(X)\equiv n-1\pmod{2}$.
\end{proof}
\medskip

\noindent
Now we can prove\\

\begin{Theorem}
\label{Theorem 2.6}
Let $r_0(f)$ be the number of branches of the reduced local curve $f=0$. Then 
$$\mu_0(f)+r_0(f)-1\equiv 0\pmod{2}.$$
\end{Theorem}
\smallskip

\begin{proof}
Suppose that $f$ is an irreducible power series. 
By the Weierstrass Preparation Theorem it suffices to consider the case 
where $f=Y^n+a_1(X)Y^{n-1}+\cdots+a_n(X)$ is a distinguished polynomial. 
Let $D(X)=\mbox{disc}_Yf(X,Y)$. By the classical formula for the intersection multiplicity
$i_0\left(f,\frac{\partial f}{\partial Y}\right)=\mbox{ord}D(X)$. 
Thus by Jung's lemma $i_0\left(f,\frac{\partial f}{\partial Y}\right)\equiv n-1\pmod{2}$ and by Teissier's lemma we get $\mu_0(f)=i_0\left(f,\frac{\partial f}{\partial Y}\right)-n+1\equiv 0\pmod{2}$.

The general case we get from Property \ref{2_4} ($ii$) applied to the decomposition of $f$:
$f=f_1\cdots f_r$, $r=r_0(f)$ into irreducible factors $f_i$.
\end{proof}
\bigskip

\noindent
For any reduced power series $f\in\mathbb{C}\{X,Y\}$ we put 
$$\delta_0(f)=\frac{1}{2}(\mu_0(f)+r_0(f)-1)$$
and call $\delta_0(f)$ the \emph{double point number} of the local curve $f=0$. \\
From the properties of the Milnor number we get\\

\begin{Proposition}\ \\[-0.3cm]
\begin{enumerate}
\item[$(i)$] $\delta_0(f)\geq 0$ is an integer, $\delta_0(f)=0$ if and only if $f=0$ is nonsingular,
\item[$(ii)$] $\delta_0(f\circ \Phi)=\delta_0(f)$ for any local system of coordinates $\Phi$,
\item[$(iii)$] $\delta_0(\prod_{i=1}^{m}{f_i})=\sum_{i=1}^{m}{\delta_0(f_i)}+\sum_{1\leq i<j\leq m}{i_0(f_i,f_j)}$
where $f_i$ are coprime\\[0.25cm] power series.
\end{enumerate}
\label{Proposition 2_7}
\end{Proposition}
\smallskip

\begin{Remark}
\label{Remark 2_8}
\emph{
The reduced curve $f=0$ has an \emph{ordinary $r$-fold singularity} if it has $r$ branches, all nonsingular and intersecting each other with multiplicity $1$. For such a curve we have $\mu_0=(r-1)^2$ and $\delta_0=\frac{1}{2}r(r-1)$.
}
\end{Remark}
\smallskip

Assume that $f\in\mathbb{C}\{X,Y\}$ is a power series with no multiple factors. If $f=f_1\cdots f_r$ is a product of irreducible factors  $f_i\in\mathbb{C}\{X,Y\}$ then we set
$$c_i(f)=\mu_0(f_i)+\sum_{j\neq i}{i_0(f_i,f_j)} \ \mbox{ for } \ i=1,\ldots,r$$
A curve $\Psi=0$ is said to be an \emph{adjoint to} $f=0$ if 
$$i_0(f_i,\Psi)\geq c_i(f) \ \mbox{ for } \ i=1,\ldots,r.$$

\begin{Remark}
\label{Remark 2_9}
\emph{
Let $f=0$ be an ordinary $r$-fold singularity. Then $\Psi=0$ is an adjoint to $f=0$ if and only if $\mbox{ord}\Psi\geq r-1$.
}
\end{Remark}
\smallskip

The following result is known as Noether's Theorem on the double-point divisor. Let $g,h\in\mathbb{C}\{X,Y\}$.\\

\begin{Theorem}
\label{Noether's Theorem on the double-point divisor}
Suppose that the local curves $f=0$ and $g=0$ have no common component. 
If $h$ satisfies Noether's conditions
$$i_0(f_i,h)\geq i_0(f_i,g)+c_i(f) \ \mbox{ for } \ i=1,\ldots,r$$
then $h$ belongs to the ideal generated by  $f,g$ in the ring \ $\mathbb{C}\{X,Y\}$.
\end{Theorem}

Let us write $h=\Phi f+\Psi g$ \ with \ $\Phi,\Psi\in\mathbb{C}\{X,Y\}$. Then Noether's conditions imply that $\Psi=0$ is an adjoint to $f=0$. In connection with Noether's Theorem  let us note\\
\begin{Theorem}
\label{Theorem 2_11}
Let $f\in\mathbb{C}\{X,Y\}$ be an irreducible power series. 
Then there does not exist $\Psi\in\mathbb{C}\{X,Y\}$ such that $i_0(f,\Psi)=\mu_0(f)-1$. Let $h\in\mathbb{C}\{X,Y\}$ be such that $i_0(f,h)=i_0(f,g)+\mu_0(f)-1$, then $h\notin(f,g)\mathbb{C}\{X,Y\}$.
\end{Theorem}
\medskip

The second part of (\ref{Theorem 2_11}) follows easily from the first.
Indeed, if we had $h=\Phi f+\Psi g$ with $\Phi,\Psi\in\mathbb{C}\{X,Y\}$ and $i_0(f,h)=i_0(f,g)+\mu_0(f)-1$ then we would get $i_o(f,\Psi)=\mu_0(f)-1$, a contradiction with the first part of (\ref{Theorem 2_11}).

Let us pass now to the proofs of Theorems (\ref{Noether's Theorem on the double-point divisor}) and (\ref{Theorem 2_11}).\\
Let $F(u,Y),G(u,Y),H(u,Y)\in\mathbb{C}\{u\}[Y]$ where $u$ is a variable. Assume that $F(u,Y)=\prod_{i=1}^{n}{(Y-y_i(u))}$ in $\mathbb{C}\{u\}[Y]$ and $y_i(u)\neq y_j(u)$ for $i\neq j$.\\

\begin{Lemma}
\label{Lemma 2_12}
If \emph{$\mbox{ord}H(u,y_i(u))\geq\mbox{ord}\frac{\partial F}{\partial Y}(u,y_i(u))G(u,y_i(u))$ for $i=1,\ldots,n$}, then
\emph{$H(u,Y)\in(F(u,Y),G(u,Y))\mathbb{C}\{u\}[Y]$}.
\end{Lemma}
\smallskip

\begin{proof}
Let
$$
\Psi(u,Y)=\sum_{i=1}^{n}{\frac{H(u,y_i(u))}{\frac{\partial F}{\partial Y}(u,y_i(u))G(u,y_i(u))}\frac{F(u,Y)}{(Y-y_i(u))}}.
$$
Then $\Psi(u,Y)\in\mathbb{C}\{u\}[Y]$ and $H(u,y_i(u))=\Psi(u,y_i(u))G(u,y_i(u))$ for \linebreak 
$i=1,\ldots,n$. 
Therefore $H(u,Y)\equiv\Psi(u,Y)G(u,Y)\bmod(Y-y_i(u))$ for \linebreak 
$i=1,\ldots,n$  and $H(u,Y)\equiv\Psi(u,Y)G(u,Y)\bmod F(u,Y)$  
what implies $H(u,Y)\in(F(u,Y),G(u,Y))\mathbb{C}\{u\}[Y]$.
\end{proof}
\smallskip

\begin{Lemma}
\label{Lemma 2_13}
If \ $\Psi(u,Y)=\Psi_0(u)Y^{n-1}+\cdots+\Psi_{n-1}(u)\in\mathbb{C}\{u\}[Y]$, then 
$$
\sum_{i=1}^{n}{\frac{\Psi(u,y_i(u))}{\frac{\partial F}{\partial Y}(u,y_i(u))}}=\Psi_0(u).
$$
\end{Lemma}
\smallskip

\begin{proof}
The lemma follows immediately from the Lagrange interpolation formula.
\end{proof}
\smallskip

\noindent
\emph{Proof of Theorem} \ref{Noether's Theorem on the double-point divisor} (cf.\cite{Waer39}, Achtes Kapitel).\\
We may assume that $f_i=f_i(X,Y)$ are $Y$-distinguished polynomials and (after replacing $g$, $h$ by te rests of division by $f$) $g,h\in\mathbb{C}\{X\}[Y]$.\\
We have\\
$i_0(f_i,g)+c_i(f)=i_0(f_i,g)+\mu_0(f_i)+\sum_{j\neq i}{i_0(f_i,f_j)}=i_0(f_i,g)-i_0(f_i,X)+1+$\\
$+i_0\left(f_i,\frac{\partial f_i}{\partial Y}\right)+\sum_{j\neq i}{i_0(f_i,f_j)}=i_0(f_i,g)-i_0(f_i,X)+1+i_0\left(f_i,\frac{\partial f}{\partial Y}\right)$\\
by Teissier's lemma.\\
Let $n_i=i_0(f_i,X)$ for $i=1,\ldots,r$. 
The Noether's conditions are equivalent to 
\begin{enumerate}
\item[$(1)$] $i_0(f_i,h)\geq i_0(f_i,g)+i_0\left(f_i,\frac{\partial f}{\partial Y}\right)-n_i+1 \ \mbox{ for } \ i=1,\ldots,r$.
\end{enumerate}
By Puiseux' Theorem we can write
$$
f_i(t^{n_i},Y)=(Y-y_{i1}(t))\cdots(Y-y_{in_i}(t)) \ \mbox{ in } \ \mathbb{C}\{t\}[Y]
$$
where $y_{i1}(t),\ldots,y_{in_i}(t)$ are $\mathbb{C}\{t^{n_i}\}$-conjugate i. e. $y_{ij}(t)=y_{i1}(\varepsilon_j t)$ for some $\varepsilon_j$ such that $\varepsilon_j^{n_i}=1$.
Thus for every $h(X,Y)\in\mathbb{C}\{X,Y\}$: \\
$$
\mbox{ord}h(t^{n_i},y_{i1}(t))=\cdots=\mbox{ord}h(t^{n_i},y_{in_i}(t))=i_0(f_i,h)
$$
and we can rewrite $(1)$ in the form
\begin{enumerate}
\item[$(2)$] $\mbox{ord}h(t^{n_i},y_{ij}(t))\geq\mbox{ord}g(t^{n_i},y_{ij}(t))+\mbox{ord}\frac{\partial f}{\partial Y}(t^{n_i},y_{ij}(t))-n_i+1$\\[0.1cm]
or else
\item[$(3)$]  $\mbox{ord}(t^{n_i-1}h(t^{n_i},y_{ij}(t))\geq\mbox{ord}g(t^{n_i},y_{ij}(t))\frac{\partial f}{\partial Y}(t^{n},y_{ij}(t))$.\\
\end{enumerate}
Let $N=n_1\cdots n_r$ and $\overline{y}_{ij}(u)=y_{ij}(u^{N/n_i})$ for $i=1,\ldots,r$.
Obviously $\frac{N}{n_i}(n_i-1)\leq N-1$ therefore $(3)$ implies 
 \begin{enumerate}
\item[$(4)$] $\mbox{ord}(u^{N-1}h(u^N,\overline{y}_{ij}(u)))\geq\mbox{ord}g(u^{N},\overline{y}_{ij}(u))\frac{\partial f}{\partial Y}(u^{N},\overline{y}_{ij}(u))$
\end{enumerate}
and we can apply Lemma \ref{Lemma 2_12} to the polynomials
$$F(u,Y)=f(u^N,Y)=\prod{(Y-\overline{y}_{ij}(u))}, \ G(u,Y)=g(u^N,Y)$$
and $H(u,Y)=u^{N-1}h(u^N,Y).$\\
We get
$$
u^{N-1}H(u^N,Y)\in(f(u^N,Y),g(u^N,Y))\mathbb{C}\{u\}[Y].
$$
It is easy to check that $\mathbb{C}\{u\}[Y]=\sum_{i=0}^{N-1}{\mathbb{C}\{u^N\}Y^i}$ is a free $\mathbb{C}\{u^N\}[Y]$-module, so
$$
h(u^N,Y)\in(f(u^N,Y),g(u^N,Y))\mathbb{C}\{u\}[Y]
$$
and consequently $h(X,Y)\in(f(X,Y),g(X,Y)\mathbb{C}\{X\}[Y]$.\\
\rightline{$\square$}\\
\bigskip

\noindent
\emph{Proof of Theorem \ref{Theorem 2_11}}.\\
Suppose that there is a $\Psi=\Psi(X,Y)\in\mathbb{C}\{X,Y\}$ such that
 \begin{enumerate}
\item[$(5)$] $i_0(f,\Psi)=\mu_0(f)-1$.
\end{enumerate}
We may assume that $f=f(X,Y)$ is a $Y$-distinguished polynomial of degree $n\geq 1$ 
and $\Psi\in\mathbb{C}\{X\}[Y]$ a polynomial of $Y$-degree $\leq n-1$. 
By Teissier's lemma we can rewrite (5) in the form
 \begin{enumerate}
\item[$(6)$] $i_0(f,X\Psi)=i_0(f,\frac{\partial f}{\partial Y})$.
\end{enumerate}
By Puiseux' Theorem we have $f(u^n,Y)=\prod_{\varepsilon^n=1}{(Y-y(\varepsilon u))}$. \\
Then $(6)$ is equivalent to
\begin{enumerate}
\item[$(7)$] $\mbox{ord}u^n\Psi(u^n,y(u))=\mbox{ord}\frac{\partial f}{\partial Y}(u^n,y(u))$.
\end{enumerate}
By $(7)$ we can write $\mbox{in}(u^n\Psi(u^n,y(u)))=c_1u^N$  ($c_1\neq 0$)
and $\mbox{in}\frac{\partial f}{\partial Y}(u^n,y(u))=c_2u^N$ 
($c_2\neq 0$) 
where 
$N=i_0\left(f,\frac{\partial f}{\partial Y}\right)$.
Therefore we get
$$
\hspace*{-2.4cm}(8) \hspace{2cm} \mbox{in}\frac{u^n\Psi(u^n,y(\varepsilon u))}{\frac{\partial f}{\partial Y}(u^n,y(\varepsilon u))}=\frac{c_1\varepsilon^Nu^N}{c_2\varepsilon^Nu^N}=c, \quad c=\frac{c_1}{c_2}.
$$
On the other hand by Lemma \ref{Lemma 2_13} applied to $\Psi(u^n,Y)$ and $f(u^n,Y)$ we have
$$
\hspace*{-3.8cm}(9) \hspace{2cm} \sum_{\varepsilon^n=1}{\frac{u^n\Psi(u^n,y(\varepsilon u))}{\frac{\partial f}{\partial Y}(u^n,y(\varepsilon u))}}=u^n\Psi_0(u^n).
$$
A contradiction, because the left side of (9) is of order zero by (8).\\
\rightline{$\square$}
\bigskip

\noindent
In what follows we need \\

\begin{Lemma}
\label{Lemma 2_14}
Let $f\in\mathbb{C}\{X,Y\}$ be an irreducible power series. Then for any integer $a\in\mathbb{Z}$ there exists power series $\phi,\psi\in\mathbb{C}\{X,Y\}$ such that $a=i_0(f,\phi)-i_0(f,\psi)$.
\end{Lemma}
\smallskip

\begin{proof}
Let $(x(t),y(t))$ be a good parametrization of the branch  $f=0$. Then the rings $\mathbb{C}\{x(t),y(t)\}$ and $\mathbb{C}\{t\}$ have the same field of fractions (see \cite{JoPf00}, Theorem 5.1.3.). Then $t^a=\frac{\phi(x(t),y(t))}{\psi(x(t),y(t))}$ for some $\phi,\psi\in\mathbb{C}\{X,Y\}$ and taking orders gives $a=i_0(f,\phi)-i_0(f,\psi)$.
\end{proof}
\smallskip

\begin{Theorem}
\label{Theorem 2_15}
The semigroup $\Gamma(f)$ of the branch $f=0$ contains all integers greater than or equal to the Milnor number $\mu_0(f)$. 
The number $\mu_0(f)-1$ does not belong to $\Gamma(f)$.
\end{Theorem}
\smallskip

\begin{proof}
Let $a$ be an integer such that $a\geq \mu_0(f)$. By Lemma \ref{Lemma 2_14} we can write $a=i_0(f,\phi)-i_0(f,\psi)$ for some $\phi,\psi\in\mathbb{C}\{X,Y\}$. Then $i_0(f,\phi)=i_0(f,\psi)+a\geq i_0(f,\psi)+\mu_0(f)$ and by Noether's Theorem $\phi=Af+B\psi$ for some $A,B\in\mathbb{C}\{X,Y\}$. Thus $a=i_0(f,Af+B\psi)-i_0(f,\psi)=i_0(f,B)\in \Gamma(f)$ and we are done.\\
The second part of \ref{Theorem 2_15} follows immediately from Theorem \ref{Theorem 2_11}. \\ 
\end{proof}
\smallskip

\noindent Using Theorem \ref{Theorem 2_15} and Property \ref{2_4} $(ii)$ we get

\begin{Theorem}
\label{Theorem 2_16}
The Milnor number is an invariant of singularity.\\
\end{Theorem}

\noindent
\textbf{Notes}\\
Milnor introduced and studied $\mu$ in the general case of isolated hipersurface
singularities in his celebrated book \cite{Mil68}. 
A topological treatment of the Milnor number in the case of plane curve singularities is given in \cite{Wall04}.
The invariant $\delta$ was defined in algebraical terms by Hironaka in \cite{Hi57}. 
The formula $2\delta=\mu + r - 1$which served in our approach as definition of $\delta$ was proved in \cite{Mil68}
by topological methods and in \cite{Ri71} on an algebraic way. 
The classical texts \cite{Ju23} and \cite{Waer39} where the Milnor number is implicit were very helpful
when writing this article.
Teissier's lemma has interesting generalizations involving the jacobian 
(see the articles by L\^{e} Dung Trang and Greuel quoted in \cite{Te76}). 
For application of the Milnor number to singularities of plane algebraic curves see \cite{GwPl01}, \cite{Mas01}
and the references given therein.\\

\section{Newton diagrams and power series}
\label{s3}

Let $\mathbb{R}_+=\{a\in\mathbb{R}\colon a\geq 0\}$. 
For any subsets $E,F\subset\mathbb{R}_+^2$ we consider the 
\emph{Minkowski sum} $E+F=\{u+v\colon u\in E \mbox{ and } v\in F\}$.
Let $E\subset\mathbb{N}^2$ and let us denote by $\Delta(E)$ the convex hull of the set $E+\mathbb{R}_+^2$. 
A subset $\Delta\subset\mathbb{R}_+^2$ is a \emph{Newton diagram} (or \emph{polygon}) 
if there is a set $E\subset \mathbb{N}^2$ such that $\Delta=\Delta(E)$. 
The smallest set $E_0\subset\mathbb{N}^2$ such that $\Delta=\Delta(E_0)$ 
is called the set of \emph{vertices} of the \emph{Newton diagram} $\Delta$. 
It is always finite and we can write $E_0=\{v_0,v_1,\ldots,v_m\}$ 
where $v_i=(\alpha_i,\beta_i)$ and $\alpha_{i-1}<\alpha_i$, $\beta_{i-1}>\beta_i$ for all $i=1,\ldots,m$. 
The Newton diagram with one vertex $v=(\alpha,\beta)$ is the quadrant  $(\alpha,\beta)+\mathbb{R}_+^2$. 
According to Teissier (see \cite{Te76}, \cite{Te91}) for two positive integers $a,b>0$ we denote by 
{\scriptsize
$
\left\{
{\setlength\arraycolsep{2pt}
\begin{array}{c}
a\\
\hline\hline
b
\end{array}}
\right\}
$}
the Newton diagram with vertices $(0,b)$ and $(a,0)$.
We denote also 
{\scriptsize
$
\left\{
{\setlength\arraycolsep{2pt}
\begin{array}{c}
a\\
\hline\hline
\infty
\end{array}}
\right\}
$}
resp.
{\scriptsize
$
\left\{
{\setlength\arraycolsep{2pt}
\begin{array}{c}
\infty\\
\hline\hline
b
\end{array}}
\right\}
$}
the quadrant with vertex $(a,0)$ resp. $(0,b)$.\\
\begin{figure}[!h]
\centering
%{\includegraphics[angle=-90,width=0.5\textwidth]{Fig_1.eps}} %Figure 1
%\subfloat[Figure 1]{
\scalebox{0.3}[0.3]{\includegraphics{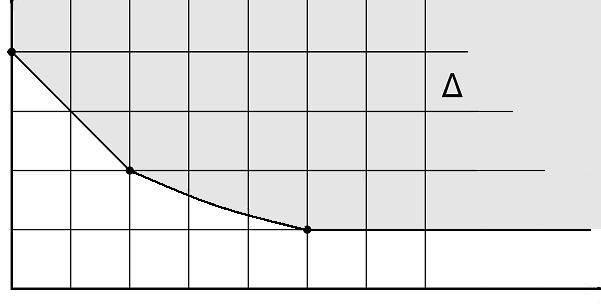}} %Figure 1
%}
\caption{}
\label{Figure 1}
\end{figure}

\noindent
We call a segment $E\subset\mathbb{R}_+^2$ Newton's edge if its vertices $(\alpha,\beta)$, $(\alpha',\beta')$ lie in $\mathbb{N}^2$ and $\alpha<\alpha'$, $\beta'<\beta$. 
We put $|E|_1=\alpha'-\alpha$ and $|E|_2=\beta-\beta'$ and call $|E|_1/|E|_2$ the inclination of $E$.
We denote by $a(E)$ and $b(E)$ the distances of $E$ to the vertical and horizontal axes respectively.\\
\begin{figure}[!h]

\centering
%{\includegraphics[angle=-90,width=0.5\textwidth]{Fig_2.eps}} %Figure 2
\scalebox{0.3}[0.3]{\includegraphics{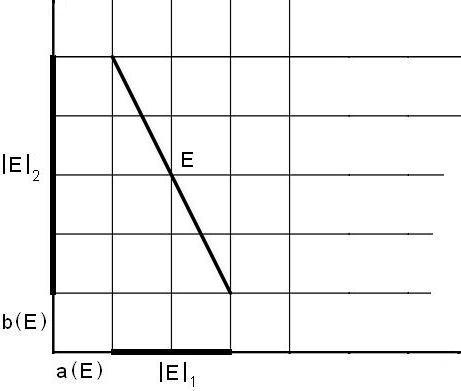}} %Figure 2
\caption{}
\label{Figure 2}
\end{figure}

\noindent
The vertices of Newton's edge $E$ are $(a(E),|E|_2+b(E))$ and $(a(E)+|E|_1,b(E))$.
For any Newton diagram $\Delta$ we consider the set $\mbox{\Large{\emph{n}}}(\Delta)$ 
of $1$-dimensional compact faces of the boundary of $\Delta$.
Note that $\mbox{\Large{\emph{n}}}(\Delta)=\emptyset$ if and only if $\Delta$ is a quadrant. 
If $\Delta$ has vertices $v_0,\ldots,v_m$ $(m>0)$ then $\mbox{\Large{\emph{n}}}(\Delta)=\{E_1,\ldots,E_m\}$ 
where $E_i$ is the edge with vertices $v_{i-1}, v_i$.

Let $a(\Delta)$ and $b(\Delta)$ denote the distances of $\Delta$ to the vertical and horizontal axes respectively. 
The diagram is \emph{convenient} if $a(\Delta)=b(\Delta)=0$. % and nearly convenient if $a(\Delta),b(\Delta)\leq 1$. 
The reader will check the following two properties of Newton diagrams.\\

\begin{Property}
\label{Property 3_1}
The Newton diagrams form a semigroup with respect to the Minkowski sum. For any Newton diagram $\Delta$ we have the \textbf{minimal decomposition} 
$$
(*) \hspace*{2cm} \Delta =
{\scriptsize
\left\{
{\setlength\arraycolsep{2pt}
\begin{array}{c}
a(\Delta)\\
\hline\hline
\infty
\end{array}}
\right\}}
+
\sum_{S\in\mbox{\large{n}}(\Delta)}
{
{\scriptsize
\left\{
{\setlength\arraycolsep{2pt}
\begin{array}{c}
|S|_1\\
\hline\hline
|S|_2
\end{array}}
\right\}}
+
{\scriptsize
\left\{
{\setlength\arraycolsep{2pt}
\begin{array}{c}
\infty\\
\hline\hline
b(\Delta)
\end{array}}
\right\}}
}.
$$
\end{Property}
\smallskip

\begin{Property}
\label{Property 3_2}
The line with the slope $-1/\theta$ $(\theta>0)$ supporting the Newton diagram $\Delta$ 
with the minimal decomposition $(*)$ intersects the horizontal axis in the point with abscissa
$$
a(\Delta) +
\sum_{S\in\mbox{\large{n}}(\Delta)}
{
\inf\{|S|_1,\theta|S|_2\}+\theta b(\Delta).
}
$$
\end{Property}
\smallskip

\begin{figure}[!h]

\centering
\scalebox{0.3}[0.3]{\includegraphics{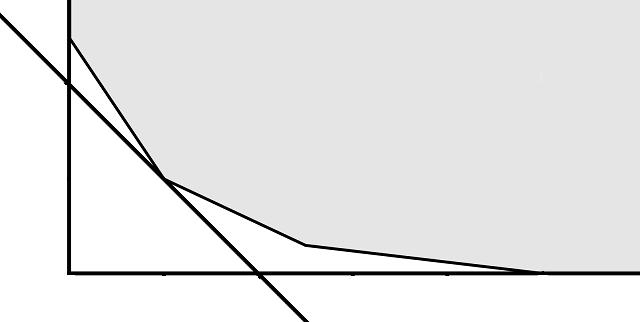}} %Figure 3
\caption{\ $\alpha+\theta\beta=\nu$}
\label{Figure 3}
\end{figure}

\noindent
For any nonzero power series $f=\sum{c_{\alpha\beta}X^\alpha Y^\beta\in\mathbb{C}\{X,Y\}}$ we put 
$\Delta(f)=\Delta(\mbox{supp}f)$ and  $\mathcal{N}(f)=\mbox{\Large\emph{n}}(\Delta(f))$. 
We call $\Delta(f)$ the Newton diagram of $f$. 
The following property is basic\\

\begin{Property}
\label{Property 3_3}
For any nonzero power series $f,g\in\mathbb{C}\{X,Y\}$:
$$
\Delta(fg)=\Delta(f)+\Delta(g)
$$
\end{Property}
\smallskip

\noindent
For the proof of \ref{Property 3_3} we refer the reader to \cite{Lip88}.\\

In particular if $g=f\cdot \mbox{unit}$ then $\Delta(g)=\Delta(f)$ since $\Delta(u)=\mathbb{R}^2_+$
(the zero of the semigroup of Newton's diagrams) 
if $u(0)\neq 0$ and we may speak about the Newton diagram of the local curve $f=0$.
A \emph{power series} $f$ is \emph{convenient} if the diagram $\Delta(f)$ is convenient.
Obviously $f$ is convenient if and only if the branches of the curve $f=0$ are different from axes.

Observe that $\mbox{\Large\emph{n}}(f)=\emptyset$ if and only if $f=X^{\alpha_0}Y^{\beta_0}\cdot \mbox{unit}$.
Suppose that $\mbox{\Large\emph{n}}(f)\neq \emptyset$. For any face $S\in\mbox{\Large\emph{n}}(f)$ we consider the initial part $\mbox{in}(f,S)$ of $f$ corresponding to $S$:
$$
\mbox{in}(f,S)=\sum_{(\alpha,\beta)\in S}{c_{\alpha\beta}X^\alpha Y^\beta}
$$
Note that $\mbox{\Large\emph{n}}(\mbox{in}(f,S))=S$. If $a(S)$ and $b(S)$ are the distances of $S$ to the axes then
$X^{a(S)}Y^{b(S)}$ is the monomial of maximal degree dividing $\mbox{in}(f,S)$.
Let $r(S)=\mbox{gcd}(|S|_1,|S|_2)$.
Then $r(S)=\#(S\cap\mathbb{N}^2)-1$.
Let $m_S=|S|_1/r(S)$, $n_S=|S|_2/r(S)$. 
It is easy to check that 
$$
\mbox{in}(f,S)=X^{a(S)}Y^{b(S)}\Phi_S(X^{m_S},Y^{n_S})
$$
where $\Phi_S(U,V)\in\mathbb{C}[U,V]$ is a homogeneous form of degree $r(S)$ such that $\Phi_S(U,0)\Phi_S(0,V)\neq 0$ in $\mathbb{C}[U,V]$.
Therefore we may write
$$
\mbox{in}(f,S)=cX^{a(S)}Y^{b(S)}\prod_{i=1}^{r}{(Y^{n_S}-a_iX^{m_S})^{d_i}}
$$
where $a_i\neq a_j$ for $i\ne j$, $c\ne 0$ are constants.\\

We put $r(f,S)=r$. Since $r(S)=\sum_{i=1}^{r}{d_i}$ we have $r(f,S)\leq r(S)$ with equality if and only if $d_1=\ldots=d_r=1$.
We say that $f$ is \emph{nondegenerate on $S$} if $r(f,S)=r(S)$.
A power series $f$ is nondegenerate on $S$ if and only if the system of equations
$\frac{\partial}{\partial X}\textrm{in}(f,S)=\frac{\partial}{\partial Y}\textrm{in}(f,S)=0$
has no solutions in $(\mathbb{C}\setminus\{0\})\times(\mathbb{C}\setminus\{0\})$.
The \emph{power series $f$} is \emph{nondegenerate} if it is nondegenerate on each $S\in\mbox{\Large\emph{n}}(f)$.
A binomial curve $Y^n-aX^m=0$, $\mbox{gcd}(n,m)=1$, $a\neq 0$ will be called a \emph{quasi-tangent to $f=0$} of (\emph{tangential}) \emph{multiplicity} $d$ if $Y^n-aX^m$ is a factor of multilicity $d$ of an initial form $\mbox{in}(f,S)$. 
We say that $Y^{n_S}-a_iX^{m_S}=0$, $i=1,\ldots,r$ are quasi-tangents to $f=0$ corresponding to the face $S\in\mbox{\Large\emph{n}}(f)$.\\

\begin{Remark}
\label{Remark 3_4}
\emph{
If $\mbox{ord}f(X,0)=\mbox{ord}f(0,Y)=\mbox{ord}f$ (this condition means that the axes $Y=0$ and $X=0$ are not tangent to the curve $f=0$) then
$
\Delta(f)=
{\scriptsize
\left\{
{\setlength\arraycolsep{2pt}
\begin{array}{c}
\mbox{ord}f\\
\hline\hline
\mbox{ord}f
\end{array}}
\right\}}
$
and the initial form corresponding to the unique face of $\Delta(f)$ is $\mbox{in}f=c\prod_{i=1}^{r}{(Y-a_iX)^{d_i}}$
where $a_i\neq a_j$ for $i\ne j$.
In this case the quasi-tangents to $f=0$ are usual tangents $Y-a_iX=0$, $i=1,\ldots,r$.
}
\end{Remark}
\smallskip

\begin{Remark}
\emph{
If $f\in\mathbb{C}\{X,Y\}$ is a convenient power series then the local curve $f=0$ has exactly one quasi-tangent
if and only if
$$
f=c(Y^n-aX^m)^d+\sum{c_{\alpha\beta}X^\alpha Y^\beta}, \quad \mbox{gcd}(n,m)=1
$$
where the summation is over $(\alpha,\beta)$ such that $\alpha n+\beta m>dmn$.\\
We have 
$\Delta(f)=
{\scriptsize
\left\{
{\setlength\arraycolsep{2pt}
\begin{array}{c}
i_0(f,Y)\\
\hline\hline
i_0(f,X)
\end{array}}
\right\}}=
{\scriptsize
\left\{
{\setlength\arraycolsep{2pt}
\begin{array}{c}
md\\
\hline\hline
nd
\end{array}}
\right\}}
$.
}
\label{Remark 3_5}
\end{Remark}
\medskip

\begin{Remark}
\emph{
Let $\mbox{mult}(f,\tau)$ be the tangential multiplicity of the quasi-tangent $\tau$
to the local curve $f=0$. 
We put $\mbox{mult}(f,\tau)=0$ if a binomial curve $\tau$ is not a quasi-tangent to $f=0$.
Then we have for any nonzero power series $f,g$:
$\mbox{mult}(fg,\tau)=\mbox{mult}(f,\tau)+\mbox{mult}(g,\tau)$.
}
\label{Remark 3_6}
\end{Remark}
\medskip

\noindent
\textbf{Notes}\\
An interesting algebra of the Newton diagrams is developed in \cite{Te76}.
Newton introduced his diagrams to solve equations $f(X,Y)=0$ (see \cite{BriKn86}).
The notion of nondegeneracy appeared in a very general setting in \cite{Kou76}, \cite{Kh77}, \cite{Va76}.
The authors are responsible for the term "quasi-tangent".\\

\section{Newton transformations and factorization of power series}
\label{sec: 4.Newton's transformations and factorization of Power series}

Let $n,m>0$ be coprime integers and let $c\neq 0$ be a complex number. 
The Newton transformation (in short: the N-transformation) is defined by the following equations\\[0.1cm]
$
(10)\hspace{1cm}
\begin{array}{l}
X=X_1^n,\\[0.1cm]
Y=(c+Y_1)X_1^m
\end{array}
$\\[0.1cm]
where $(X_1,Y_1)$ are new variables.\\
The N-transformation (10) may be viewed as a deformation of the parametrization\\[0.1cm]
$
(11)\hspace{1cm}
\begin{array}{l}
X=X_1^n,\\[0.1cm]
Y=cX_1^m
\end{array}
$\\[0.1cm]
of the binomial curve $Y^n-c^nX^m=0$.\\
We omit the simple proof of the following
\medskip

\begin{Lemma}
\label{Lemma 4_1}
Let $f=f(X,Y)\in\mathbb{C}\{X,Y\}$ be a nonzero power series without constant term. 
Then there is a unique power series $f_1=f_1(X_1,Y_1)\in\mathbb{C}\{X_1,Y_1\}$ and an integer $k>0$ such that
$$
f(X_1^n,(c+Y_1)X_1^m)=X_1^kf_1(X_1,Y_1), \ f_1(0,Y_1)\neq 0
$$
in $\mathbb{C}\{X_1,Y_1\}$.\\

The line $\alpha n+\beta m=k$ is a supporting line of $\Delta(f)$.
Moreover, the series $f_1$ is without constant term if and only if the curve $Y^n-c^nX^m=0$ is a quasi-tangent
to curve $f=0$.
Its tangential multiplicity equals $i_0(f_1,X_1)=\mbox{ord}(f_1(0,Y_1))$.\\
\end{Lemma}

In what follows we call $f_1=f_1(X_1,Y_1)$ the strict transform of the series $f=f(X,Y)$ by the N-transformation (10).\\

The lemma below gives a necessary condition for a power series to be irreducible.
\smallskip

\begin{Lemma}
\label{Lemma 4_2}
Let $f=f(X,Y)\in\mathbb{C}\{X,Y\}$ be a convenient irreducible power series. 
Then the local curve $f=0$ has exactly one quasi-tangent.
\end{Lemma}
\smallskip

\begin{proof}
Let $N=\mbox{ord}f(0,Y)$, $M=\mbox{ord}f(X,0)$.
By the Weierstrass Preparation Theorem we have 
$f=(Y^N+a_1(X)Y^{N-1}+\cdots+a_N(X))\cdot$unit. 
By Puiseux' Theorem $Y^N+a_1(t^N)Y^{N-1}+\cdots+a_N(t^N)=\prod_{\varepsilon^N=1}{(Y-y(\varepsilon t))}$
where $y(t)\in\mathbb{C}\{t\}$.
A simple calculation shows that $\mbox{ord}a_i(X)\geq i\frac{M}{N}$ with equality for $i=N$.
Therefore we have 
$
\Delta(f)=
{\scriptsize
\left\{
{\setlength\arraycolsep{2pt}
\begin{array}{c}
M\\
\hline\hline
N
\end{array}}
\right\}}
$.
Let $I=\{i\in[1,N]\colon \mbox{ord}a_i=i\frac{M}{N}\}$.
Then the initial form of $f$ corresponding to the unique face of $\Delta(f)$ is equal to $\mbox{const.}(Y^N+\sum_{i\in I}\mbox{in}a_i(X)Y^{N-i})=\mbox{const.}(Y^n-c^nX^m)^d$ where $N=nd$, $M=md$ and $\mbox{in}y(t)=ct^M$.
This proves the lemma.
\end{proof}
\smallskip

We can use the N-transformations to decide in a finite number of steps if a power series is irreducible.
\smallskip

\begin{Lemma}
\label{Lemma 4_3}
Suppose that $f=f(X,Y)\in\mathbb{C}\{X,Y\}$ is a convenient power series such that the curve $f=0$ 
has exactly one quasi-tangent $Y^n-c^nX^m=0$.
Let $f_1=f_1(X_1,Y_1)\in\mathbb{C}\{X_1,Y_1\}$ be the strict transform of $f=f(X,Y)$ by the N-transformation $(10)$.
Then $f$ is irreducible if and only if $f_1$ is irreducible.
\end{Lemma}
\smallskip

\begin{proof}
Let $d$ be the tangential multiplicity of the quasi-tangent  $Y^n-c^nX^m=0$.
Then $\mbox{ord}f_1(0,Y_1)=d$.
First assume that $f=f(X,Y)$ is an irreducible power series.
Let $(t^e,\varphi(t))$ be a Puiseux' parametrization of an irreducible factor of $f_1(X_1,Y_1)$.
Then $e\leq \mbox{ord}f_1(0,Y_1)=d$.
On the other hand, by the definition of the strict transform we get $f(t^{en},ct^{em}+t^{em}\varphi(t))=0$ in $\mathbb{C}\{t\}$.
Since $f$ is irreducible we get $en\geq \mbox{ord}f(0,Y)=dn$ and $e\geq d$.
Thus $e=d$ and $f_1$ is an irreducible power series.

To check that the irreducibility of $f_1$ implies the irreducibility of $f$ assume that $f_1$ is irreducible.
Then the branch $f_1=0$ has a Puiseux' parametrization $(t^d,\varphi(t))$ where $d=\mbox{ord}f_1(0,Y)$.
By the definition of the strict transform we get
$f(t^{dn},ct^{dm}+t^{dm}\varphi(t))=0$ in $\mathbb{C}\{t\}$.
Since $\mbox{ord}f(0,Y)=dn$ it suffices to check that $(t^{dn},ct^{dm}+t^{dm}\varphi(t))$ is a Puiseux' parametrization.
We have $\mbox{gcd}(dn,\mbox{supp}(ct^{dm}+t^{dm}\varphi(t)))=1$ since $\mbox{gcd}(d,\mbox{supp}\varphi(t))=1$.
Therefore the power series $f$ is irreducible.
\end{proof}
\smallskip

\begin{Corollary}
\label{Corollary 4_4}
Every power series $f$ of the form $f=Y^n-aX^m+\sum{c_{\alpha\beta}X^\alpha Y^\beta}$, $\textrm{gcd}(n,m)=1$
where the summation is over $(\alpha,\beta)$ such that $\alpha n+\beta m>nm$ is irreducible.
\end{Corollary}
\smallskip

\begin{proof}
The strict transform $f_1$ of $f$ by the N-transformation $(10)$ with $c$ such that $c^n=a$ is of order $1$ since $\frac{\partial f_1}{\partial Y_1}(0,0)\neq 0$.
Therefore by Lemma \ref{Lemma 4_3} the series $f$ is irreducible.
\end{proof}
\smallskip

\begin{Example}[see \cite{Kuo89}] \ \\
\emph{
Let $f=(X^2-Y^3)^2-Y^7$ and $g=(X^2-Y^3)^2-XY^5$.
The both series have the unique quasi-tangent $Y^3-X^2=0$ (of tangential multiplicity  2). 
The strict transforms of $f$ and $g$ by the N-transformation $X=X_1^3$, $Y=(1+Y_1)X_1^2$ are $f_1(X_1,Y_1)=(3Y_1)^2-X_1^2+$ terms of order $>2$ and $g_1(X_1,Y_1)=-X_1+$ terms of order $>1$.
Thus by Lemma \ref{Lemma 4_3} the series $f$ is reducible (since $f_1$ has two tangents) and the series $g$ is irreducible.
}\label{Example 4_5}
\end{Example}
\smallskip

With any binomial curve $\tau\colon Y^n-aX^m=0$, $a\neq 0$, $\mbox{gcd}(n,m)=1$ we associate the N-transformation
$$
\begin{array}{l}
X=X_\tau^n\\[0.1cm]
Y=(a^{1/n}+Y_\tau)X_\tau^m
\end{array}
$$
where $(X_\tau,Y_\tau)$ are new variables and $a^{1/n}=|a|^{1/n}\exp(i\frac{\alpha}{n})$ 
if $a=|a|\exp(i\alpha)$ with $0\leq \alpha< 2\pi$.
We denote by $f_\tau=f_\tau(X_\tau,Y_\tau)$ the strict transform of $f=f(X,Y)$ 
by the N-transformation associated with $\tau$.\\

The following property follows easily from the definitions.
\smallskip

\begin{Property}
\label{Property 4_6}
Let $f,g$ be nonzero power series without constant term.
Then
\begin{enumerate}
\item[$(i)$] a binomial curve $\tau$ is a quasi-tangent to the curve $f=0$ if and only if $f_\tau(0)=0$,
\item[$(ii)$] $(fg)_\tau=f_\tau g_\tau$ for any binomial curve $\tau$,
\item[$(iii)$] if $f=f_1\cdots f_r$ is a decomposition of $f$ into irreducible factors then for any binomial curve $\tau$: $\tau$ is a quasi-tangent to the curve $f=0$ if and only if $\tau$ is a quasi-tangent to a branch $f_i=0$ for some $i\in\{1,\ldots,r\}$.
\end{enumerate}
\end{Property}
\medskip

If $f=f_1\cdots f_r$ is a decomposition of a nonzero power series  $f$ 
without constant term into irreducible factors then we put $r_0(f)=r$ 
i. e. $r_0(f)$ is the number of irreducible factors of $f$ counted with multiplicities.
\medskip

\begin{Proposition}
\label{Proposition 4_7}
If $f\in\mathbb{C}\{X,Y\}$ is a convenient power series then $r_0(f)=\sum_{\tau}{r_0(f_\tau)}$ 
where the summation is over all quasi-tangents $\tau$ to the curve $f=0$.
\end{Proposition}
\medskip

\begin{proof}
Let $f=f_1\cdots f_r$ be a factorization of $f$ into irreducible factors $f_i$.
Let $\tau$ be a quasi-tangent to $f=0$ and let 
$$
I_\tau=\{i\in[1,r]\colon \mbox{the branch $f_i=0$ has the quasi-tangent $\tau$}\}.
$$
Thus the sets $I_\tau$ are nonempty, pairwise disjoint and \ $\displaystyle {\bigcup_{\tau}{I_\tau}=[1,r]}$. \\
Let $I_\tau^c=[1,r]\setminus I_\tau$. 
By Property \ref{Property 4_6} $(ii)$ we get
$f_\tau=\prod_{i\in I_\tau}{(f_i)_\tau}\cdot\mbox{unit}$
since $(f_i)_\tau(0)\neq 0$ for $i\in I_\tau^c$.
Therefore we obtain $r_0(f_\tau)=\sum_{i\in I_\tau}{r_0((f_i)_\tau)}=\#I_\tau$ since
$r_0((f_i)_\tau)=1$ for $i\in I_\tau$ by Lemma \ref{Lemma 4_2} and we have 
$\sum_{\tau}{r_0(f_\tau)}=\sum_{\tau}{(\# I_\tau)}=\#[1,r]=r=r_0(f)$.
\end{proof}
\bigskip

\noindent For any convenient power series $f\in\mathbb{C}\{X,Y\}$ we put
\begin{eqnarray*}
r(f,\Delta(f))=\sum_{S\in\mbox{\large\emph{n}}(f)}{r(f,S)}&=&
\mbox{the number of quasi-tangents to the curve }\\[-0.4cm]
&&f=0,\\[0.1cm]
r(\Delta(f))=\sum_{S\in\mbox{\large\emph{n}}(f)}{r(S)}&=&\mbox{the number of quasi-tangents counted with}\\[-0.4cm]
&&\mbox{tangential multiplicities to the curve } f=0
\end{eqnarray*}

Obviously, $r(f,\Delta(f))\leq r(\Delta(f))$ with equality if and only if $f$ is nondegenera\-te. 
Note also that $r(\Delta(f))\!=\!\!\mbox{ the number of integral points lying on}
\bigcup\mbox{\Large\emph{n}}(f)-1$. Hence the integral points divide $\bigcup\mbox{\Large\emph{n}}(f)$ into $r(\Delta(f))$ segments.\\

\begin{Proposition}
\label{Proposition 4_8}
For any convenient power series $f\in\mathbb{C}\{X,Y\}$ we have
$$
r(f,\Delta(f))\leq r_0(f)\leq r(\Delta(f)).
$$
If $f$ is nondegenerate then $r_0(f)=r(\Delta(f))$ 
and the quasi-tangents to the branches of the local curve $f=0$ have tangential multiplicity equal to $1$.\linebreak
Different branches have different quasi-tangents.
\end{Proposition}
\medskip

\begin{proof}
By Proposition \ref{Proposition 4_7} we have $r_0(f)=\sum_{\tau}{r_0(f_\tau)}$. 
Therefore $r_0(f,\Delta(f))= \sum_{\tau}{1}\leq r_0(f)\leq\sum_{\tau}{\mbox{ord}f_\tau}\leq\sum_{\tau}{\mbox{ord}f_\tau(0,Y_\tau)}=r(\Delta(f))$ 
since $\mbox{ord}f_\tau(0,Y_\tau)$ equals the tangential multiplicity of $\tau$ (by Lemma \ref{Lemma 4_1}) and the number of quasi-tangents counted with multiplicities associated with the face $S$ is equal to $r(S)$.
Suppose that $f$ is nondegenerate. 
Then $r(f,\Delta(f))=r(\Delta(f))$ and $r(f)=r(\Delta(f))$ by the first part of the proposition.
We have $\mbox{mult}(f,\tau)=\prod_{i=1}^{r}{\mbox{mult}(f_i,\tau)}$ by Remark \ref{Remark 3_6} and the assertion 
about the branches of the local curve $f=0$ follows.
\end{proof}
\smallskip

\begin{Example}
\emph{
Let $f=X^7+X^5Y+X^3Y^2+2X^2Y^3+XY^4+Y^6$.
Then the local curve $f=0$ has four quasi-tangents: $Y^2+X=0$, 
$Y-\varepsilon X^2=0$, $Y-\overline{\varepsilon}X^2=0$, ($\varepsilon^2+\varepsilon+1=0$), $Y+X=0$. 
The quasi-tangent $\tau\colon Y+X=0$ is of tangential multiplicity $2$, 
the remaining quasi-tangents are of tangential multiplicity $1$.
Then $4\leq r_0(f)\leq 5$. By Proposition \ref{Proposition 4_8} we have $r_0(f)=3+r_0(f_\tau)$.
To calculate $r_0(f_\tau)$ we use the N-transformation $X=X_\tau$, $Y=(-1+Y_\tau)X_\tau$.
We get $f(X_\tau,(-1+Y_\tau)X_\tau)=X_\tau^5f_\tau(X_\tau,Y_\tau)$ where
$f_\tau=-X_\tau+$higher order terms. We have $\mbox{ord}f_\tau=1$ and $f_\tau$ is irreducible i. e. $r_0(f_\tau)=1$. 
Consequently $r_0(f)=3+1=4$.\\
}
\label{Example 4_9}
\end{Example}
\medskip

\noindent
\textbf{Notes}:\\
Although the Newton transformations appear when using the Newton algorithm
\cite{Du89}, \cite{JoPf00}, \cite{Mau80} 
a systematic treatment of this notion was given quite recently in \cite{CNVe12}.\\

\section{Newton transformations, intersection multiplicity and the Milnor number}
\label{sec: 5.Newton's transformations, intersection multiplicity and the Milnor number}

The Minkowski double area $[\Delta,\Delta']\in\mathbb{N}\cup\{\infty\}$ of the pair $\Delta,\Delta'$ 
of Newton diagrams is uniquely determined by the following conditions
\begin{enumerate}
\item[(M1)] $[\Delta_1+\Delta_2,\Delta']=[\Delta_1,\Delta']+[\Delta_2,\Delta']$,
\item[(M2)] $[\Delta,\Delta']=[\Delta',\Delta]$,
\item[(M3)]
$
\left[
{\scriptsize
\left\{
{\setlength\arraycolsep{2pt}
\begin{array}{c}
a\\
\hline\hline
b
\end{array}}
\right\}},
{\scriptsize
\left\{
{\setlength\arraycolsep{2pt}
\begin{array}{c}
a'\\
\hline\hline
b'
\end{array}}
\right\}}
\right]
=\inf\{ab',a'b\}.
$
\end{enumerate}
\medskip

\begin{Lemma}
\label{Lemma 5_1}
If $\Delta=\sum_{S\in\mbox{\large{n}}(\Delta)}
{
{\scriptsize
\left\{
{\setlength\arraycolsep{2pt}
\begin{array}{c}
|S|_1\\
\hline\hline
|S|_2
\end{array}}
\right\}}
}
$
is a convenient Newton diagram then
\begin{enumerate}
\item[$(i)$] $[\Delta,\Delta]=2\ \mbox{ area } \ (\mathbb{R}_+^2\setminus\Delta)$
\item[$(ii)$] $\displaystyle{[\Delta,\Delta]=\sum_{S\in\mbox{\large{n}}(\Delta)}{(|S|_1|S|_2+a(S)|S|_2+b(S)|S|_1)}}$.
\end{enumerate}
\end{Lemma}
\medskip

\begin{proof}
By (M1) and (M3) we get $[\Delta,\Delta]=\sum_{S,T\in\mbox{\emph{\large{n}}}(\Delta)}{\inf\{|S|_1|T|_2,|S|_2|T|_1)\}}$
which implies $(i)$.\\
To check $(ii)$ observe that $|S|_1|S|_2+a(S)|S|_2+b(S)|S|_1$ equals to the double area of triangle with vertices 
$(0,0)$, $(a(S),|S|_2+b(S))$ and $(a(S)+|S|_1,b(S))$ and use $(i)$.
\end{proof}
\medskip

\begin{Lemma}
\label{Lemma 5_2}
Let $\Delta$ be a Newton diagram. 
Then for every Newton's edge $E$ the supporting line of $\Delta$ parallel to $E$ is described by the equation
$$
|E|_2\alpha+|E|_1\beta=
\left[
{\scriptsize
\left\{
{\setlength\arraycolsep{2pt}
\begin{array}{c}
|E|_1\\
\hline\hline
|E|_2
\end{array}}
\right\}},
\Delta
\right]
$$
\end{Lemma} 
\medskip

\begin{proof}
The lemma follows from Property \ref{Property 3_2} by putting $\theta=\frac{|E|_1}{|E|_2}$ 
into the formula for the abscissa of the point at which the supporting line intersects the axis $\beta=0$.
\end{proof}
\smallskip

\begin{Theorem}
\label{Theorem 5_3}
Let $f$ be a nonzero power series without constant term.
Then for every convenient power series $h$:
$$
i_0(f,h)=[\Delta(f),\Delta(h)]+\sum_{\tau}{i_0(f_\tau,h_\tau)}
$$
where the summation is over all quasi-tangents $\tau$ to the curve $h=0$.
\end{Theorem}
\smallskip

\begin{proof}
Fix a nonzero power series $f$ without constant term.
It is easy to check that if the theorem is true for two power series $h_1$, $h_2$ then it is true for their product $h_1h_2$.
Thus it suffices to prove the theorem for irreducible power series $h$.

Let $h$ be a convenient irreducible power series and let $\tau\colon Y^n-aX^m=0$ 
be the unique quasi-tangent to the branch $h=0$ of tangential multiplicity $d$.
Let $c=a^{1/n}$.
Then
$\Delta(h)=
{\scriptsize
\left\{
{\setlength\arraycolsep{2pt}
\begin{array}{c}
dn\\
\hline\hline
dm
\end{array}}
\right\}}
$,
$h(X^n_\tau,(c+Y_\tau)X^m_\tau)=X^{dmn}_\tau h_\tau(X_\tau,Y_\tau)$,
$\mbox{ord}h_\tau(0,Y_\tau)=d$ and $f(X^n_\tau,(c+Y_\tau)X^m_\tau)=X^{k}_\tau f_\tau(X_\tau,Y_\tau)$,
$f_\tau(0,Y_\tau)\neq 0$ in $\mathbb{C}\{Y_\tau\}$.

Let $(t^d,\varphi(t))$ be a Puiseux' parametrization of $h_\tau(X_\tau,Y_\tau)=0$.\\
Then $(t^{dn},(c+\varphi(t)t^{dm})$ is a Puiseux' parametrization of $h(X,Y)=0$ and we have\\
$i_0(f,h)=\mbox{ord}f(t^{dn},(c+\varphi(t))t^{dm})=dk+\mbox{ord}f_\tau(t^d,\varphi(t))=dk+i_0(f_\tau,h_\tau)=$\\
$=[\Delta(h),\Delta(f)]+i_0(f_\tau,h_\tau)$
by Lemma \ref{Lemma 5_2} since $\alpha dn+\beta dm=dk$ is the supporting line of $\Delta(f)$ parallel
to the unique face of $\Delta(h)$.
\end{proof}
\medskip

\begin{Example}
\label{Example 5_4}
\emph{
Let $f=Y^3+X^4Y-X^7$, $g=XY-(X^2+Y^2)^2$. 
Then 
$
\Delta(f)=
{\scriptsize
\left\{
{\setlength\arraycolsep{2pt}
\begin{array}{c}
4\\
\hline\hline
2
\end{array}}
\right\}}
+
{\scriptsize
\left\{
{\setlength\arraycolsep{2pt}
\begin{array}{c}
3\\
\hline\hline
1
\end{array}}
\right\}}
$, 
$\Delta(g)=
{\scriptsize
\left\{
{\setlength\arraycolsep{2pt}
\begin{array}{c}
1\\
\hline\hline
3
\end{array}}
\right\}}
+
{\scriptsize
\left\{
{\setlength\arraycolsep{2pt}
\begin{array}{c}
3\\
\hline\hline
1
\end{array}}
\right\}}
$
and $[\Delta(f),\Delta(g)]=\inf\{4\cdot 3,1\cdot 2\}+\inf\{4\cdot 1,2\cdot 3\} 
+\inf\{3\cdot 3,1\cdot 1\}+\inf\{3\cdot 1,1\cdot 3\}=10$.
The local curves $f=0$ and $g=0$ have exactly one common quasi-tangent
$\tau\colon Y-X^3=0$. 
The N-transformation associated with $\tau$ is $X=X_\tau$, $Y=(1+Y_\tau)X^3_\tau$.
A simple calculation shows that $f_\tau=Y_\tau+(1+Y_\tau)^3X^2_\tau$
and
$g_\tau=Y_\tau-2X^4_\tau(1+Y_\tau)^2-(1+Y_\tau)^3X^5_\tau$.
Thus we have 
$\Delta(f_\tau)=
{\scriptsize
\left\{
{\setlength\arraycolsep{2pt}
\begin{array}{c}
2\\
\hline\hline
1
\end{array}}
\right\}}
$,
$
\Delta(g_\tau)=
{\scriptsize
\left\{
{\setlength\arraycolsep{2pt}
\begin{array}{c}
4\\
\hline\hline
1
\end{array}}
\right\}}
$,
the local curves $f_\tau=0$, $g_\tau=0$ have no common quasi-tangent and
$i_0(f_\tau,g_\tau)=[\Delta(f_\tau),\Delta(g_\tau)]=\inf\{2\cdot 1,1\cdot 4\}=2$.
By Theorem \ref{Theorem 5_3} we get
$i_0(f,g)=[\Delta(f),\Delta(g)]+[\Delta(f_\tau),\Delta(g_\tau)]=10+2=12$.
}
\end{Example}
\medskip

A pair of power series $f,h$ is \emph{nondegenerate} if the local curves $f=0$, \linebreak
$h=0$ have no common quasi-tangent.
It is easy to check that the pair $f,h$ is nondegenerate if and only if for $S\in\mbox{\emph{\Large{n}}}(f)$ and $T\in\mbox{\emph{\Large{n}}}(h)$ we have the following:
\begin{enumerate}
\item[(a)] either $S$ and $T$ are not parallel, i. e. $|S|_1|T|_2\neq |S|_2|T|_1$, or
\item[(b)] the faces $S$ and $T$ are parallel and the system of equations $\mbox{in}(f,S)(X,Y)=0$, 
$\mbox{in}(h,T)(X,Y)=0$ has no solutions in $(\mathbb{C}\setminus\{0\})\times(\mathbb{C}\setminus\{0\})$.
\end{enumerate}
\medskip

\begin{Corollary}[see \cite{Ber75}, \cite{Kh79}]\ \\
Let $f,h$ be nonzero power series without constant term.
Suppose that $f$ or $h$ is convenient.
Then $i_0(f,h)\geq [\Delta(f),\Delta(h)]$ with equality if and only if the pair $(f,h)$ is nondegenerate.
\label{Corollary 5_5}
\end{Corollary}
\medskip

For any convenient Newton diagram $\Delta=\sum_{S\in\mbox{\emph{\large{n}}}(\Delta)}
{
{\scriptsize
\left\{
{\setlength\arraycolsep{2pt}
\begin{array}{c}
|S|_1\\
\hline\hline
|S|_2
\end{array}}
\right\}}
}
$
we put $|\Delta|_1=\sum_{S\in\mbox{\emph{\large{n}}}(\Delta)}{|S|_1}$,
$|\Delta|_2=\sum_{S\in\mbox{\emph{\large{n}}}(\Delta)}{|S|_2}$.
Then $\Delta$ intersects the axes in points $(0,|\Delta|_2)$ and $(|\Delta|_1,0)$.\\

\begin{Theorem}
\label{Theorem 5_6}
For any convenient power series $h\in\mathbb{C}\{X,Y\}$: 
$$
i_0\left(h,\frac{\partial h}{\partial Y}\right)=[\Delta(h),\Delta(h)]-|\Delta(h)|_1+\sum_{\tau}i_0\left(h_\tau,\frac{\partial h_\tau}{\partial Y_\tau}\right)
$$
where the summation is over all quasi-tangents $\tau$ to the local curve $h=0$.
\end{Theorem}
\medskip

\begin{proof}
We may assume that $\frac{\partial h}{\partial Y}(0,0)=0$. We will check that
\begin{enumerate}
\item[$(i)$] $[\Delta(h),\Delta\left(\frac{\partial h}{\partial Y}\right)]=[\Delta(h),\Delta(h)]-|\Delta(h)|_1$,
\item[$(ii)$]  if $\tau$ is a quasi-tangent to the curve $h=0$ then $\left(\frac{\partial h}{\partial Y}\right)_\tau=\frac{\partial h_\tau}{\partial Y_\tau}$.
\end{enumerate}
Proof of $(i)$. 
We have 
$[\Delta(h),\Delta\left(\frac{\partial h}{\partial Y}\right)]=
\sum_{S\in\mbox{\emph{\large{n}}}(\Delta)}{\left[
{\scriptsize
\left\{
{\setlength\arraycolsep{2pt}
\begin{array}{c}
|S|_1\\
\hline\hline
|S|_2
\end{array}}
\right\}},
\Delta\left(\frac{\partial h}{\partial Y}\right)
\right]}$.
The line $\alpha|S|_2+\beta|S|_1={\left[
{\scriptsize
\left\{
{\setlength\arraycolsep{2pt}
\begin{array}{c}
|S|_1\\
\hline\hline
|S|_2
\end{array}}
\right\}},
\Delta\left(\frac{\partial h}{\partial Y}\right)
\right]}$ 
supporting the diagram 
$\Delta\left(\frac{\partial h}{\partial Y}\right)$
and parallel to $S$ passes through the point
$(a(S),|S|_2+b(S)-1)$.
Thus we have
$a(S)|S|_2+(|S|_2+b(S)-1)|S|_1=
\left[
{\scriptsize
\left\{
{\setlength\arraycolsep{2pt}
\begin{array}{c}
|S|_1\\
\hline\hline
|S|_2
\end{array}}
\right\}},
\Delta\left(\frac{\partial h}{\partial Y}\right)
\right]
$
and we get
$
\left[
\Delta(h),\Delta\left(\frac{\partial h}{\partial Y}\right)
\right]=
\sum_{S}(a(S)|S|_2+|S|_2|S|_1+b(S)|S|_1-|S|_1)=
[\Delta(h),\Delta(h)]-|\Delta(h)|_1
$
by Lemma \ref{Lemma 5_1} $(ii)$.\\
Proof of $(ii)$. Let $\tau\colon Y^n-aX^m=0$ be a quasi-tangent to the curve $h=0$. 
Let $c=a^{1/n}$ We have
$$
(\ )\qquad h(X^n_\tau,(c+Y_\tau)X^m_\tau)=X^k_\tau h_\tau(X_\tau,Y_\tau)\quad\mbox{in}\ \ \mathbb{C}\{X_\tau,Y_\tau\}
$$
where $\alpha n+\beta m=k$ is a supporting line of $\Delta(h)$.\\
Then $k>m$ and the line $\alpha n+\beta m=k-m$ supports the diagram $\Delta\left(\frac{\partial h}{\partial Y}\right)$.
Differentiating $(\ )$ with respect to $Y_\tau$ we get 
$\frac{\partial h}{\partial Y}(X^n_\tau,(c+Y_\tau)X_\tau^m)X_\tau^m=X_\tau^k\frac{\partial h_\tau}{\partial Y_\tau}(X_\tau,Y_\tau)$
and
$\frac{\partial h}{\partial Y}(X^n_\tau,(c+Y_\tau)X_\tau^m)=X_\tau^{k-m}\frac{\partial h_\tau}{\partial Y_\tau}(X_\tau,Y_\tau)$.
Therefore $\left(\frac{\partial h}{\partial Y}\right)_\tau=\frac{\partial h_\tau}{\partial Y_\tau}$
and $(ii)$ follows.\\
Now we have by Theorem \ref{Theorem 5_3} and properties $(i)$, $(ii)$:\\[0.2cm]
$\displaystyle{
i_0\left(h,\frac{\partial h}{\partial Y}\right)=\left[\Delta(h),\Delta\left(\frac{\partial h}{\partial Y}\right)\right]+
\sum_{\tau}{i_0\left(h_\tau,\left(\frac{\partial h}{\partial Y}\right)_\tau\right)}=
}
$\\[0.2cm]
$
\displaystyle{
=[\Delta(h),\Delta(h)]-|\Delta(h)|_1+\sum_{\tau}{i_0\left(h_\tau,\frac{\partial h_\tau}{\partial Y_\tau}\right)}.
}
$
\end{proof}
\medskip

\noindent
For any convenient Newton diagram $\Delta$ we put\\[0.2cm]
$\mu(\Delta)=[\Delta,\Delta]-|\Delta|_1-|\Delta|_2+1$,\\[0.2cm]
$\delta(\Delta)=\frac{1}{2}(\mu(\Delta)+r(\Delta)-1)$.\\

\begin{Theorem}
\label{Theorem 5_7}
Let $f\in\mathbb{C}\{X,Y\}$ be a convenient power series. Then
\begin{enumerate}
\item[$(i)$] $\mu_0(f)=\mu_0(\Delta(f))+r(\Delta(f))+\sum_{\tau}{(\mu_0(f_\tau)-1)}$,
\item[$(ii)$] $\displaystyle{\delta_0(f)=\delta(\Delta(f))+\sum_{\tau}{\delta_0(f_\tau)}}$
\end{enumerate}
where the summation is over all quasi-tangents $\tau$ to the local curve $f=0$.
\end{Theorem}
\medskip

\begin{proof}
$(i)$ By Teissier's lemma and Theorem \ref{Theorem 5_6} applied to the power series $f$ we get\\
$
\begin{array}{rcl}
\mu_0(f)&=&i_0\left(f,\frac{\partial f}{\partial Y}\right)-i_0(f,X)+1=\\[0.1cm]
&=&[\Delta(f),\Delta(f)]-|\Delta(f)|_1-|\Delta(f)|_2+1+\sum_{\tau}{i_0\left(f_\tau,\frac{\partial f_\tau}{\partial Y_\tau}\right)}=\\[0.1cm]
&=&\mu(\Delta(f))+\sum_{\tau}{i_0\left(f_\tau,\frac{\partial f_\tau}{\partial Y_\tau}\right)}=\\[0.1cm]
&=&\mu(\Delta(f))+\sum{(\mu_0(f_\tau)+i_0(f_\tau,X)-1)}=\\[0.1cm]
&=&\mu(\Delta(f))+r(\Delta(f))+\sum_{\tau}{(\mu_0(f_\tau)-1)}.
\end{array}
$\\[0.3cm]
$
\begin{array}{lcl}
(ii) \quad 2\delta_0(f)&=&\mu_0(f)+r_0(f)-1=\\[0.1cm]
&=&\mu(\Delta(f))+r(\Delta(f))+\sum_{\tau}{(\mu_0(f_\tau)-1)}+\sum_{\tau}{r_0(f_\tau)}-1=\\[0.1cm]
&=&\mu(\Delta(f))+r(\Delta(f))-1+\sum_{\tau}{(\mu_0(f_\tau)+r_0(f_\tau)-1)}=\\[0.1cm]
&=&2\delta(\Delta(f))+2\sum_{\tau}{\delta_0(f_\tau)}\\[0.1cm]
\end{array}
$\\[0.1cm]
and $(ii)$  follows.
\end{proof}
\bigskip

\noindent
We can rewrite the formula \ref{Theorem 5_7} $(i)$ for the Milnor number in the form
$$
\mu_0(f)=\mu(\Delta(f))+r(\Delta(f))-r(f,\Delta(f))+\sum_{\tau}{\mu_0(f_\tau)}.
$$
Then we get\\

\begin{Corollary}
\label{Corollary 5_8}
For any convenient power series $f\in\mathbb{C}\{X,Y\}$:
\begin{enumerate}
\item[$(i)$] $\mu_0(f)\geq \mu(\Delta(f))+r(\Delta(f))-r(f,\Delta(f))$ 
with equality if and only if all strict transforms $f_\tau$ corresponding to the quasi-tangents $\tau$ of the curve $f=0$ are nonsingular,
\item[$(ii)$] $\delta_0(f)\geq \delta(\Delta(f))$ with equality if and only if $\mu_0(f)= \mu(\Delta(f))+r(\Delta(f))-r(f,\Delta(f))$.
\end{enumerate}
\end{Corollary}
\medskip

\begin{Corollary}[Planar Kouchnirenko's Theorem, see \cite{Kou76}] \ \\
For any convenient power series $f\in\mathbb{C}\{X,Y\}$:\\
$\mu_0(f)\geq \mu(\Delta(f))$ with equality if and only if $f$ is nondegenerate.\\
\label{Corollary 5_9}
\end{Corollary}

\begin{Example}\emph{(cf. Example \ref{Example 4_9})\\
Let $f=X^7+X^5Y+X^3Y^2+2X^2Y^3+XY^4+X^6$ and $\Delta=\Delta(f)$.\\
Then $\mu(\Delta)=18$, $r(\Delta)=5$, $r(f,\Delta)=4$, $\delta(\Delta)=11$.
All strict transforms of $f$ by N-transformations corresponding to the quasi-tangents to $f=0$ are nonsingular.
Therefore we get $\mu_0(f)=\mu(\Delta)+r(\Delta)-r(f,\Delta)=19$ and $\delta_0(f)=\delta(\Delta)=11$.\\
\label{Example 5_10}
}
\end{Example}

\begin{Example}\emph{(cf. Example \ref{Example 4_5})\\
Let $f=(X^2-Y^3)^2-Y^7$ and $g=(X^2-Y^3)^2-XY^5$.
Then 
$\Delta(f)=\Delta(g)=\Delta=$\\
$=
{\scriptsize
\left\{
{\setlength\arraycolsep{2pt}
\begin{array}{c}
6\\
\hline\hline
4
\end{array}}
\right\}}
$,
$\mu(\Delta)=15$, $r(f,\Delta)=r(g,\Delta)=1$ and $r(\Delta)=2$.\\
The strict transforms $f_1$ and $g_1$ of $f$ and $g$ by the N-transformation corresponding to the unique quasi-tangent
$Y^3-X^2=0$ of $f$ and $g$ are $f_1=(3Y_1)^2-X_1^2+\cdots$ and $g_1=-X_1+\cdots$ . 
Therefore we get 
$\mu_0(f)=\mu(\Delta)+r(\Delta)-r(f,\Delta)+\mu_0(f_1)=$\\$=16+\mu_0(f_1)=16+1=17$
and
$\mu_0(g)=\mu(\Delta)+r(\Delta)-r(f,\Delta)+\mu_0(g_1)=$\\$=16+\mu_0(g_1)=16+0=16$.\\
}
\label{Example 5_11}
\end{Example}
\medskip

We say that a local curve $f=0$ is in a general position with respect to coordinates $(X,Y)$ 
if the axes $Y=0$ and $X=0$ are not tangent to $f=0$ (see Remark \ref{Remark 3_4}).
Let $t_0(f)$ be the number of tangents to the curve $f=0$.\\

\begin{Corollary}[see \cite{Ca00} and appendix to \cite{Ph73}] \ \\
Suppose that the local curves  $f=0$ and  $g=0$ are in a general position with respect to $(X,Y)$.
Let $n=\emph{ord}f$ and $m=\emph{ord}g$.\\
Then
\begin{enumerate}
\item[$(i)$] $i_0(f,g)=nm+\sum_{\tau}{i_0(f_\tau,g_\tau)}$,
\item[$(ii)$] $\delta_0(f)=\frac{1}{2}n(n-1)+\sum_{\tau}{\delta_0(f_\tau)}$,
\item[$(iii)$] $i_0(f,\frac{\partial f}{\partial Y})=n(n-1)+\sum_{\tau}{i_0(f_\tau,\frac{\partial f_\tau}{\partial Y_\tau})}$,
\item[$(iv)$] $\mu_0(f)+t_0(f)-1=n(n-1)+\sum_{\tau}{\mu_0(f_\tau)}$.
\end{enumerate}
\label{Corollary 5_12}
\end{Corollary}
\medskip

\noindent
\textbf{Notes}\\
The formulae for the local invariants in terms of the Newton diagrams
and Newton transformations (Theorems \ref{Theorem 5_3}, \ref{Theorem 5_6}, \ref{Theorem 5_7}) 
are very close to Gwo\'{z}dziewicz' formulae \cite{Gw10} in toric geometry
of plane curve singularities (see also \cite{Oka96}) 
and like Newton trees and Newton process developed by Pi. Cassou-Nogu\`{e}s and Veys in \cite{CNVe12}
provide an effective method of calculations.
The Newton number $\mu(\Delta)$ can be defined for all Newton diagrams $\Delta$
in such a way that Kouchnirenko's theorem holds for any reduced power series
(see \cite{GaLenPl07}, \cite{Len08}, \cite{Wall99}.
Corollary 5.8 provides a new characterization of weakly Newton nondegenerate singularities
(see \cite{GrNg10}, Theorem 3.3).\\

\section{Nondegenerate singularities and equisingularity}
\label{sec: Nondegenerate singularities and equisingularity}

We will prove in this section that the equisingularity class of the curve $f=0$ can be recovered form the Newton diagram $\Delta(f)$
provided that $f$ is a convenient and nondegenerate power series.\\

\begin{Lemma}
\label{Lemma 6_1}
Let $f\in\mathbb{C}\{X,Y\}$ be a convenient power series such that the curve $f=0$ has exactly one quasi-tangent.
If its tangential multiplicity is equal to $1$ then $f$ is irreducible and 
$\Gamma(f)=i_0(f,X)\mathbb{N}+i_0(f,Y)\mathbb{N}$.\\
\end{Lemma}

\begin{proof}
Let $m=i_0(f,Y)$, $n=i_0(f,X)$. 
Then $\mbox{gcd}(m,n)=1$ and after multiplying $f$ by a constant we may assume (see Remark \ref{Remark 3_5}) that
$f=Y^n+aX^m+\sum{c_{\alpha\beta}X^\alpha Y^\beta}$ where $a \ne 0$ and the summation is over $(\alpha,\beta)$ such that 
$\alpha n + \beta m>mn$.
The power series $f$ is irreducible by Corollary \ref{Corollary 4_4}.

To prove that $\Gamma(f)=\mathbb{N}n+\mathbb{N}m$ we follow \cite{Zar73} (proof of Theorem 3.9).
Consider the intersection number $i_0(f,g)$ where $g\in\mathbb{C}\{X,Y\}$ is not a multiple of $f$.
By the Weierstrass Division Theorem we may assume that 
$g=g_0(X)+g_1(X)+\cdots+g_{n-1}(X)Y^{n-1}\in \mathbb{C}\{X\}[Y]$.
We have $i_0(f,g_k(X)Y^{n-k})=(\mbox{ord}g_k)n+(n-k)m\equiv(n-k)m\pmod{n}$.
If $k,l<n$ and $k\ne l$ then $(k-l)m\not\equiv 0\pmod{n}$.
Thus $i_0(f,g_k(X)Y^{n-k})\ne i_0(f,g_l(X)Y^{n-l})$ for $k\ne l$ 
and by Property $5$ of intersection multiplicity (Section \ref{s1})
we get $i_0(f,g)=i_0(f,g_k(X)Y^{n-k})$ for a $k\in[0,n-1]$
which implies $\Gamma(f)\subset\mathbb{N}n+\mathbb{N}m$.
Since $n,m\in\Gamma(f)$ we have  $\Gamma(f)=\mathbb{N}n+\mathbb{N}m$.
\end{proof}
\medskip

\begin{Lemma}
\label{Lemma 6_2}
Let $f\in\mathbb{C}\{X,Y\}$ be a convenient, nondegenerate power series and let
$f=\prod_{i=1}^{r}{f_i}$ with $f_i\in\mathbb{C}\{X,Y\}$ irreducible.
For any $S\in\mbox{\Large{n}}(f)$ we put 
$I(S)=\{i\in[1,r]\colon \frac{i_0(f_i,Y)}{i_0(f_i,X)}=\frac{m_S}{n_S}\}$.
Then
\begin{enumerate}
\item[$(1)$] $\Gamma(f_i)=n_S\mathbb{N}+m_S\mathbb{N}$ \ for \ $i\in I(S)$,
\item[$(2)$] $i_0(f_i,f_j)=\inf\{m_Sn_T,m_Tn_S\}$ \ for \ $(i,j)\in I(S)\times I(T)$, $i\ne j$.
\end{enumerate}
\end{Lemma}
\medskip

\begin{proof}
By Proposition \ref{Proposition 4_8} the irreducible factors $f_i$, $i=1,\ldots,r$
satisfy the assumptions of Lemma \ref{Lemma 6_1}. 
Thus $\Gamma(f_i)=i_0(f_i,X)\mathbb{N}+i_0(f_i,Y)\mathbb{N}=n_S\mathbb{N}+m_S\mathbb{N}$
(we have $i_0(f_i,X)=m_S$, $i_0(f_i,Y)=n_S$ 
since $i_0(f,X)$, $i(f_i,X)$ and $m_S=|S|_1/r(S)$, $n_S=|S|_2/r(S)$ are coprime)
and we get $(1)$.
By Proposition \ref{Proposition 4_8} the pairs $f_i,f_j$, $i\ne j$ are nondegenerate,
therefore $(2)$ follows from Corollary \ref{Corollary 5_5}.
\end{proof}
\medskip

\begin{Remark}
\emph{
The quasi-tangent to the branches $f_i=0$, $i\in I(S)$ are exactly the quasi-tangents to the curve
$f=0$ corresponding to the face $S\in\mbox{\Large{\emph{n}}}(f)$.
Therefore $\#I(S)=r(S)$ and $\bigcup_{S\in\mbox{\large{\emph{n}}}(f)}{I(S)}=[1,r]$.}
\label{Remark 6_3}\\
\end{Remark}

\begin{Theorem}[see \cite{Len08}]
\label{Theorem 6_4}\ \\
Let $f,g\in\mathbb{C}\{X,Y\}$ be convenient power series such that $\Delta(f)=\Delta(g)$.
Then 
\begin{enumerate}
\item[$(i)$] if $f,g$ are nondegenerate then the curves $f=0$ and $g=0$ are equisingular,
\item[$(ii)$] if $f$ is nondegenerate but $g$ is degenerate then the curves  $f=0$ and $g=0$ are not equisingular.
\end{enumerate}
\end{Theorem}
\medskip

\begin{proof}
Let $\Delta=\Delta(f)=\Delta(g)$ and $r=r(\Delta)$.\\
$(i)$ We have $r(f)=r(g)=r$ by Proposition \ref{Proposition 4_8}.
Moreover we can label the irreducible factors $f_i$ of $f$ and $g_i$ of $g$ ($i=1,\ldots, r$)
in such a way that
$$
\frac{i_0(f_i,Y)}{i_0(f_i,X)}=\frac{i_0(g_i,Y)}{i_0(g_i,X)} \quad \mbox{ for } \quad i=1,\ldots, r.
$$
Therefore $\Gamma(f_i)=\Gamma(g_i)$ for all $ i=1,\ldots, r$ and $i_0(f_i,f_j)=i_0(g_i,g_j)$ 
for $i,j\in\{1,\ldots, r\}$ by Lemma \ref{Lemma 6_2} that is $f_i\mapsto g_i$ is an equisingularity bijection 
and the curves $f=0$ and $g=0$ are equisingular.\\
$(ii)$ By Kouchnirenko's theorem (Corollary \ref{Corollary 5_9}) we have 
$\mu_0(f)=\mu(\Delta)$ and $\mu_0(g)>\mu(\Delta)$.
Therefore $\mu_0(f)\ne\mu_0(g)$ and the curves $f=0$ and $g=0$ are not equisingular by Theorem \ref{Theorem 2_16}.
\end{proof}
\medskip

\begin{Remark}
\emph{
We have proved that if $f$ and $g$ are nondegenerate then the equisingularity bijection $f_i\mapsto g_i$
preserves the intersection multiplicities of the branches with the axes:
$i_0(f_i,X)=i_0(g_i,X)$, $i_0(f_i,Y)=i_0(g_i,Y)$ for $i=1,\ldots, r$.
}\label{Remark 6_5}
\end{Remark}
\medskip

\begin{Remark}
\emph{
We can weaken the assumption "$f,g$ convenient" of Theorem \ref{Theorem 6_4} 
by assuming only that $f,g$ have no multiple factors. 
To prove this it suffices to use Theorem \ref{Theorem 6_4}  and Remark \ref{Remark 6_5}. 
}\label{Remark 6_6}
\end{Remark}
\medskip

\begin{Example}
\emph{
Let $\Delta\in\mathbb{R}^2_+$ be a convenient Newton diagram with vertices 
$(\alpha_0,\beta_0),\ldots, (\alpha_m,\beta_m)$ 
where $0=\alpha_0<\alpha_1<\cdots<\alpha_m$
and $\beta_0>\beta_1>\cdots>\beta_m=0$.\\
Then the series $f_0=X^{\alpha_0}Y^{\beta_0}+\cdots+X^{\alpha_m}Y^{\beta_m}$
is nondegenerate and $\Delta(f_0)=\Delta$.
}\label{Example 6_7}
\end{Example}
\medskip

Let us consider an invariant $I$ of equisingularity.
For every convenient Newton diagram $\Delta$ we put $I(\Delta)=I(\Delta(f))$
where $f$ is a nondegenerate power series.
According to Theorem \ref{Theorem 6_4} $I(\Delta)$ is  defined correctly (does not depend on $f$).
There is a natural problem: 
calculate $I(\Delta)$ effectively in terms of $\Delta$.
The most known result of this kind is due to Kouchnirenko, see \cite{Kou76} 
and Corollary \ref{Corollary 5_9} in this note.\\
%We leave to the reader the direct proof of the following twodimensional case 
%of Merle-Teissier's theorem (see \cite{MerTe77}):

%\begin{Theorem}
%\label{Theorem 6_8}
%Let $f$ be a convenient, nondegenerate power series. 
%The plane curve $\Psi=0$ is an adjoint to $f=0$ 
%if and only if
%$f\mbox{supp}(XY\Psi)$ is contained in the interior $\emph{int} \Delta(f)$ of the Newton diagram $\Delta(f)$.
%\end{Theorem}

\noindent
\textbf{Notes}\\
The nondegenerate plane curve singularities may be characterized without refering to the coordinates \cite{GaLenPl07}.
An unexpected example of degeneracy is discussed in \cite{Brz11}.
A lot of invariants of nondegenerate singularities  are computed in terms of their Newton diagrams:
see survey articles \cite{Pf80} and \cite{GwLenPl10}.
A description of the adjoints to the local nondegenerate hypersurface is given in \cite{MerTe77}.
The Newton diagrams and the notion of non-degeneracy are useful also in real analytic geometry \cite{Ku91}.

\bigskip
\bigskip

\noindent
Institut de Mat\'{e}matiques de Bordeaux,\\
Universit\'{e} Bordeaux I, 350, Cours de la Lib\'{e}ration,\\
33 405, Talence Cedex 05, France\\
e-mail: picassou@math.u-borrdeaux1.fr\\

\noindent
Department of Mathematics, Kielce University of Technology,\\
Al. 1000L PP 7, 25-314 Kielce, Poland\\
e-mail: matap@tu.kielce.pl

\end{document}